\documentclass[a4paper]{article}
\usepackage{amsmath,amssymb,amsthm,amsfonts,graphicx,fullpage}
\usepackage[english]{babel}
\usepackage{color}
\usepackage{url}
\usepackage{wrapfig}
\usepackage{graphicx}   
\usepackage{caption}
\usepackage{subcaption}  
\usepackage[title]{appendix} 
\usepackage{cite}
\usepackage{enumitem}
\usepackage{endnotes}
\usepackage{hyperref}
\usepackage{cleveref}

\newtheorem{assumption}{Assumption}[section]
\newtheorem{theorem}{Theorem}[section]
\newtheorem{lemma}{Lemma}[section]
\newtheorem{remark}{Remark}[section]
\newtheorem{proposition}{Proposition}[section]
\newtheorem{definition}{Definition}[section]
\newtheorem{corollary}{Corollary}[section]
\begin{document}
	\thispagestyle{empty}
	\title{Steady-state solutions for a reaction-diffusion equation with Robin boundary conditions: \\Application to the control of dengue vectors.}
	\date{\vspace{-5ex}}
	\maketitle
	\begin{center}
	    \author{L.~Almeida\footnotemark[1]\textsuperscript{,}\footnotemark[2], P.A.~Bliman\footnotemark[1]\textsuperscript{,}\footnotemark[2],
	    N.~Nguyen\footnotemark[3]\textsuperscript{,}\footnotemark[1],
	    N.~Vauchelet\footnotemark[3]}
	\end{center}
    \footnotetext[1]{MAMBA, Inria Paris; LJLL, Sorbonne University, 5 Place Jussieu, 75005 Paris, France} 
    \footnotetext[2]{CNRS University Paris Cite}
    \footnotetext[3]{LAGA, CNRS UMR 7539, Institut Galilée, Université Sorbonne Paris Nord, 99 avenue Jean-Baptiste Clément , 93430 Villetaneuse, France} 
	\begin{abstract}
		 In this paper, we investigate an initial-boundary-value problem of a reaction-diffusion equation in a bounded domain with a Robin boundary condition and introduce some particular parameters to consider the non-zero flux on the boundary. This problem arises in the study of mosquito populations under the intervention of the population replacement method, where the boundary condition takes into account the inflow and outflow of individuals through the boundary. Using phase-plane analysis, the present paper studies the existence and properties of non-constant steady-state solutions depending on several parameters. Then, we use the principle of linearized stability to prove some sufficient conditions for their stability. We show that the long-time efficiency of this control method depends strongly on the size of the treated zone and the migration rate. To illustrate these theoretical results, we provide some numerical simulations in the framework of mosquito population control. 
	\end{abstract}
	\tableofcontents
\section{Introduction}
	The study of scalar reaction-diffusion equations $\partial_t p - \Delta p = f(p)$ with a given nonlinearity $f$ has a long history. For suitable choices of $f$, this equation can be used to model some phenomena in biology such as population dynamics (see e.g. \cite{FIF}, \cite{MUR}, \cite{SMO}). To investigate the structure of the steady-state solutions, the semilinear elliptic equation $\Delta p + f(p) = 0$ has been studied extensively.
    
    Many results about the multiplicity of positive solutions for the parametrized version $\Delta p + \lambda f(p) = 0$ in a bounded domain are known. Here, $\lambda$ is a positive parameter. Various works investigated the number of solutions and the global bifurcation diagrams of this equation according to different classes of the nonlinearity $f$ and boundary conditions. For Dirichlet problems, in \cite{LIO}, Lions used many ``bifurcation diagrams" to describe the solution set of this equation with several kinds of nonlinearities $f$, and gave nearly optimal multiplicity results in each case. The exact number of solutions and the precise bifurcation diagrams with cubic-like nonlinearities $f$ were given in the works of  Korman {\it et. al.} \cite{KOR96}, \cite{KOR97}, Ouyang and Shi \cite{OUY98} and references therein. In these works, the authors developed a global bifurcation approach to obtain the exact multiplicity of positive solutions. In the case of one-dimensional space with two-point boundary, Korman gave a survey of this approach in \cite{KOR06}. Another approach was given by Smoller and Wasserman in \cite{SMO2} using phase-plane analysis and the time mapping method. This method was completed and applied in the works of Wang \cite{WAN1}, \cite{WAN2}. While the bifurcation approach is convenient to solve the problem with more general cubic nonlinearities $f$, the phase-plane method is more intuitive and easier to compute. 
    
   Although many results were obtained concerning the number of solutions for Dirichlet problems, relatively little seems to be known concerning the results for other kinds of boundary conditions. For the Neumann problem, the works of Smoller and Wasserman \cite{SMO2}, Schaaf \cite{SCHA}, and Korman \cite{KOR02} dealt with cubic-like nonlinearities $f$ in one dimension. Recently, more works have been done for Robin boundary conditions (see e.g. \cite{DAN}, \cite{SHI}, \cite{ZHA}), or even nonlinear boundary conditions (see e.g. \cite{GOD}, \cite{GOR} and references therein). However, those works only focused on other types of nonlinearities such as positive and monotone $f$. To the best of our knowledge, the study of Robin problems with cubic-like nonlinearities remains quite open.

    In this paper, we study the steady-state solutions with values in $[0,1]$ of a reaction-diffusion equation in one dimension with inhomogeneous Robin boundary conditions 
    \begin{equation}
		\begin{cases}
			\partial_t p^0 - \partial_{xx} p^0 = f(p^0) & \text{ in }  (0,\infty) \times \Omega , \\
			\frac{\partial p^0}{\partial \nu} = -D(p^0 - p^\text{ext}) & \text{ on } (0,\infty) \times \partial\Omega, \\
			p^0(0,\cdot) = p^\text{init} &  \quad\text{ in } \Omega,
		\end{cases}
		\label{eqn:pb1}
	\end{equation}
	where $\Omega = (-L,L)$ is a bounded domain in $\mathbb{R}$. The steady-state solutions satisfy the following elliptic boundary-value problem,
	\begin{equation}
		\begin{cases}
			-p'' & = f(p) \qquad \qquad \qquad  \text{ in } (-L,L), \\
 			p'(L) & = -D(p(L) - p^\text{ext}), \\
 			-p'(-L)   & = - D(p(-L) - p^\text{ext}). 
		\end{cases}
		\label{eqn:pb2}
	\end{equation} 
	where $L > 0$, $D > 0, p^\text{ext} \in (0,1)$ are constants. The reaction term $f: [0,1] \rightarrow \mathbb{R}$ is of class $\mathcal{C}^1$, with three roots $\{0, \theta, 1\}$ where $0 < \theta < 1$ (see \cref{fig:f(q)}). The dynamics of \cref{eqn:pb1} can be determined by the structure of steady-state solutions which satisfy \cref{eqn:pb2}. Note that, by changing variable from $x$ to  $y = x/L$, then \cref{eqn:pb2} becomes $p''(y) + L^2 f(p(y)) = 0$ on $(-1,1)$ with parameter $L^2$. Thus, we study problem \cref{eqn:pb2} with three parameters $L> 0, D >0$, and $p^\text{ext} \in (0,1)$.
	
	The Robin boundary condition considered in \cref{eqn:pb1} and \cref{eqn:pb2} means that the flow across the boundary points is proportional to the difference between the surrounding density and the density just inside the interval. Here we assume that $p^\text{ext}$ does not depend on space variable $x$ nor time variable $t$. 
	
	The existence of classical solutions for such problems was studied widely in the theory of elliptic and parabolic differential equations (see, for example, \cite{PAO}). In our problem, due to difficulties caused by the inhomogeneous Robin boundary condition and the variety of parameters, we cannot obtain the exact multiplicity of solutions. However, our main results in \cref{thm:exist} and \ref{thm:nonSM} show how the existence of solutions and their ``shapes'' depend on parameters $D, p^\text{ext}$ and $L$. The idea of phase plane analysis and time-map method as in \cite{SMO2} are extended to prove these results. 
	
	Since the solutions of \cref{eqn:pb2} are equilibria of \cref{eqn:pb1}, their stability and instability are the next problems that we want to investigate. The stability analysis of the non-constant steady-state solutions is a delicate problem especially when the system under consideration has multiple steady-state solutions. In \cref{thm:stability}, we use the principle of linearized stability to give some sufficient conditions for stability. Finally, as a consequence of these theorems, we obtain \cref{result} which provides a comprehensive result about existence and stability of the steady-state solutions when the size $L$ is small. 
	
	The main biological application of our results is the control of dengue vectors. {\it Aedes} mosquitoes are vectors of many vector-borne diseases, including dengue. Recently, a biological control method using an endosymbiotic bacterium called {\it Wolbachia} has gathered a lot of attention. {\it Wolbachia} helps reduce the vectorial capacity of mosquitoes and can be passed to the next generation. Massive release of mosquitoes carrying this bacterium in the field is thus considered as a possible method to replace wild mosquitoes and prevent dengue epidemics. Reaction-diffusion equations have been used in previous works to model this replacement strategy (see \cite{BAR, CHA, STR16}). In this work, we introduce the Robin boundary condition to describe the migration of mosquitoes through the boundary. Since inflows of wild mosquitoes and outflows of mosquitoes carrying {\it Wolbachia} may affect the efficiency of the method, the study of existence and stability of steady-state solutions depending on parameters $D, p^\text{ext}$ and $L$ as in \cref{eqn:pb2}, \cref{eqn:pb1} will provide necessary information to maintain the success of the control method using {\it Wolbachia} under the effects of migration. 
	
	Problem (\ref{eqn:pb1}) arises often in the study of population dynamics. $p^0$ is usually considered as the relative proportion of one population when there are two populations in competition. This is why, we only focus on solutions with values that belong to the interval $[0,1]$. \cref{eqn:pb1} is derived from the idea in paper \cite{STR16}, where the authors reduce a reaction-diffusion system modelling the competition between two populations $n_1$ and $n_2$ to a scalar equation on the proportion $p = \dfrac{n_1}{n_1 + n_2}$. More precisely, they consider two populations with a very high fecundity rate scaled by a parameter $\epsilon > 0 $ and propose the following system depending on $\epsilon$ for $t > 0, x \in \mathbb{R}^d$,
	\begin{equation}
		\begin{cases}
			\partial_t n_1^\epsilon - \Delta n_1^\epsilon = n_1^\epsilon f_1(n_1^\epsilon,n_2^\epsilon), \\
			\partial_t n_2^\epsilon - \Delta n_2^\epsilon = n_2^\epsilon f_2(n_1^\epsilon,n_2^\epsilon).
			\label{eqn:redu1}
		\end{cases}
	\end{equation}
	The authors obtained that under some appropriate conditions, the proportion $p^\epsilon = \dfrac{n_1^\epsilon}{n_1^\epsilon + n_2^\epsilon}$ converges strongly in $L^2(0,T;L^2(\mathbb{R}^d))$, and weakly in $L^2(0,T;H^1(\mathbb{R}^d))$ to the solution $p^0$ of the scalar reaction-diffusion equation $\partial_t p^0 - \Delta p^0 = f(p^0)$ when $\epsilon \rightarrow 0$ , where $f$ can be given explicitly from $f_1, f_2$. 
	
	Now, in order to describe and study the migration phenomenon, we aim here at considering system \cref{eqn:redu1} in a bounded domain $\Omega$ and introduce the boundary conditions to characterize the inflow and outflow of individuals as follows
	\begin{equation}
		\begin{cases}
			\frac{\partial n_1^\epsilon}{\partial \nu} = -D(n_1^\epsilon - n_1^{\text{ext},\epsilon}) & \text{ on } (0,T) \times \partial \Omega ,\\
			\frac{\partial n_2^\epsilon}{\partial \nu} = -D(n_2^\epsilon - n_2^{\text{ext},\epsilon}) & \text{ on }  (0,T) \times \partial \Omega,
			\label{eqn:redu2}
		\end{cases}
	\end{equation}
	where $n_1^{\text{ext},\epsilon}, n_2^{\text{ext},\epsilon}$ depend on $\epsilon$ but do not depend on time $t$ and position $x$. \cref{eqn:redu2} models the tendency of the population to cross the boundary, with rates proportional to the difference between the surrounding density and the density just inside $\Omega$. Reusing the idea in \cite{STR16}, we prove in \cref{sec:converge} that the proportion $p^\epsilon = \dfrac{n_1^\epsilon}{n_1^\epsilon + n_2^\epsilon}$ converges on any bounded time-domain to the solution of \cref{eqn:pb1} when $\epsilon$ goes to zero. Hence, we can reduce the system \cref{eqn:redu1}, \cref{eqn:redu2} to a simpler setting as in \cref{eqn:pb1}. The proof is based on a relative compactness argument that was also used in previous works about singular limits (e.g. \cite{STR16, HIL08, HIL13}), but here, the use of the trace theorem is necessary to prove the limit on the boundary. 

    The outline of this work is the following. In the next section, we present the setting of the problem and the main results. In \cref{sec:proof}, we provide detailed proof of these results. Section \ref{sec:bio} is devoted to an application to the biological control of mosquitoes. We also present numerical simulations to illustrate the theoretical results we obtained. \cref{sec:converge} is devoted to proving the asymptotic limit of a 2-by-2 reaction-diffusion system when the reaction rate goes to infinity. Finally, we end this article with a conclusion and perspectives section.

\section{Results on the steady-state solutions}
\label{sec:result}

\subsection{Setting of the problem}
	In one-dimensional space, consider the system \cref{eqn:pb1} in a bounded domain $\Omega = (-L,L) \subset \mathbb{R}$. Let $D > 0$, $p^\text{ext} \in (0,1)$ be some constant and $p^\text{init}(x) \in [0,1]$ for all $x \in (-L,L)$.
	The reaction term $f$ satisfies the following assumptions
	\begin{assumption}[bistability]
		\label{reaction}
		Function $f: [0,1] \rightarrow \mathbb{R}$ is of class $\mathcal{C}^1([0,1])$ and $f(0) = f(\theta) = f(1) = 0$ with $\theta \in (0,1)$, $f(q) < 0$ for all $q \in (0,\theta)$, and $f(q) > 0$ for all $q \in (\theta,1)$. Moreover, $\displaystyle \int_{0}^{1} f(s)ds > 0$. 
	\end{assumption}

    \begin{assumption}[convexity]
		\label{convexity}
		There exist $\alpha_1 \in (0,\theta)$ and $\alpha_2 \in (\theta,1)$ such that $f'(\alpha_1) = f'(\alpha_2) = 0$, $f'(q) < 0$ for any $q \in [0,\alpha_1) \cup (\alpha_2,1]$, and $f'(q) > 0$ for $q \in (\alpha_1,\alpha_2)$. Moreover, $f$ is convex on $(0,\alpha_1)$ and concave on $(\alpha_2,1)$.
	\end{assumption}
	A function $f$ satisfying \cref{reaction} and \cref{convexity} is illustrated in \cref{fig:f(q)}.
 	\begin{remark}
 	    \label{extreme}
 		\begin{enumerate}[label=(\alph*),leftmargin=1\parindent]
 		    \item[]
 		    \item Due to \cref{reaction} and the fact that $p^\text{ext} \in (0,1), p^\text{init}(x) \in [0,1]$ for any $x$, one has that $0$ and $1$ are respectively sub- and super-solution of problem \cref{eqn:pb1}. Since $f$ is Lipschitz continuous on $(0,1)$ then by Theorem 4.1, Section 2.4 in \cite{PAO}, we obtain that problem \cref{eqn:pb1} has a unique solution $p^0$ that is in $\mathcal{C}^{1,2}((0,T]\times \Omega)$  with $0 \leq p^0(t,x) \leq 1$ for all $x \in (-L,L), t > 0$. 
 			\item Again by Assumption \ref{reaction}, $0$ and $1$ are respectively sub- and super-solutions of \cref{eqn:pb2}. For fixed values of $D, p^\text{ext}$ and $L$, we use the same method as in \cite{PAO} to obtain that there exists a $\mathcal{C}^2$ solution of \cref{eqn:pb2} with values in $[0,1]$. However, \cref{reaction} and \cref{convexity} on $f$ are not enough to conclude the uniqueness of the solution. In the following section, we prove that the stationary problem \cref{eqn:pb2} may have multiple solutions and their existence depends on the values of the parameters.
 			\item
 			For any $p^\text{ext} \in (0,1)$ and $p^\text{ext} \neq \theta$, system \cref{eqn:pb2} cannot have a monotone solution on the whole interval $(-L,L)$. Indeed, assume that \cref{eqn:pb2} admits an increasing solution $p$ on $(-L,L)$ (the case when $p$ is decreasing on $(-L,L)$ is analogous). Thus, we have $p'(x) \geq 0$ for all $x \in [-L,L]$ and $p(L) > p(-L)$. So thanks to the boundary condition of \cref{eqn:pb2}, one has
 			\begin{center}
 			    $ D p^\text{ext} = p'(L) + Dp(L) \geq Dp(L) > Dp(-L) \geq -p'(-L) + Dp(-L) = D p^\text{ext},$
 			\end{center}
 			which is impossible. Therefore, we can deduce that the solutions of system \cref{eqn:pb2} always admit at least one locally extremum on the open interval $(-L,L)$.
  		\end{enumerate}
 	\end{remark}
 	
 	\noindent To study system \cref{eqn:pb2}, we define function $F$ (see \cref{fig:F(q)}) as below
	\begin{equation}
		F(q) = \displaystyle \int_{0}^{q} f(s)ds,
		\label{eqn:F}
	\end{equation} 
	then $F'(q) = f(q)$ and $F(0) = 0$. From \cref{reaction}, $F$ reaches the minimal value at $q = \theta$ and the (locally) maximal values at $q = 0$ and $q = 1$. Since $\displaystyle \int_0^1 f(s) ds >0$, then $F(1) > F(0)$, it implies that $F(1) = \displaystyle \max_{[0,1]} F; F(\theta) = \displaystyle \min_{[0,1]}F$.
	Moreover, since $F(\theta) < F(0)$ and function $F$ is monotone in $(\theta,1)$ ($F'(q) = f(q) > 0$ for any $q \in (\theta,1)$). Thus, there exists a unique value $\beta \in (\theta,1)$ such that 
	\begin{equation}
		F(\beta) = F(0) = 0.
		\label{eqn:beta}
	\end{equation}
	\begin{figure}
		\centering
		\begin{subfigure}{0.45\textwidth}
			\centering
			\includegraphics[scale=3]{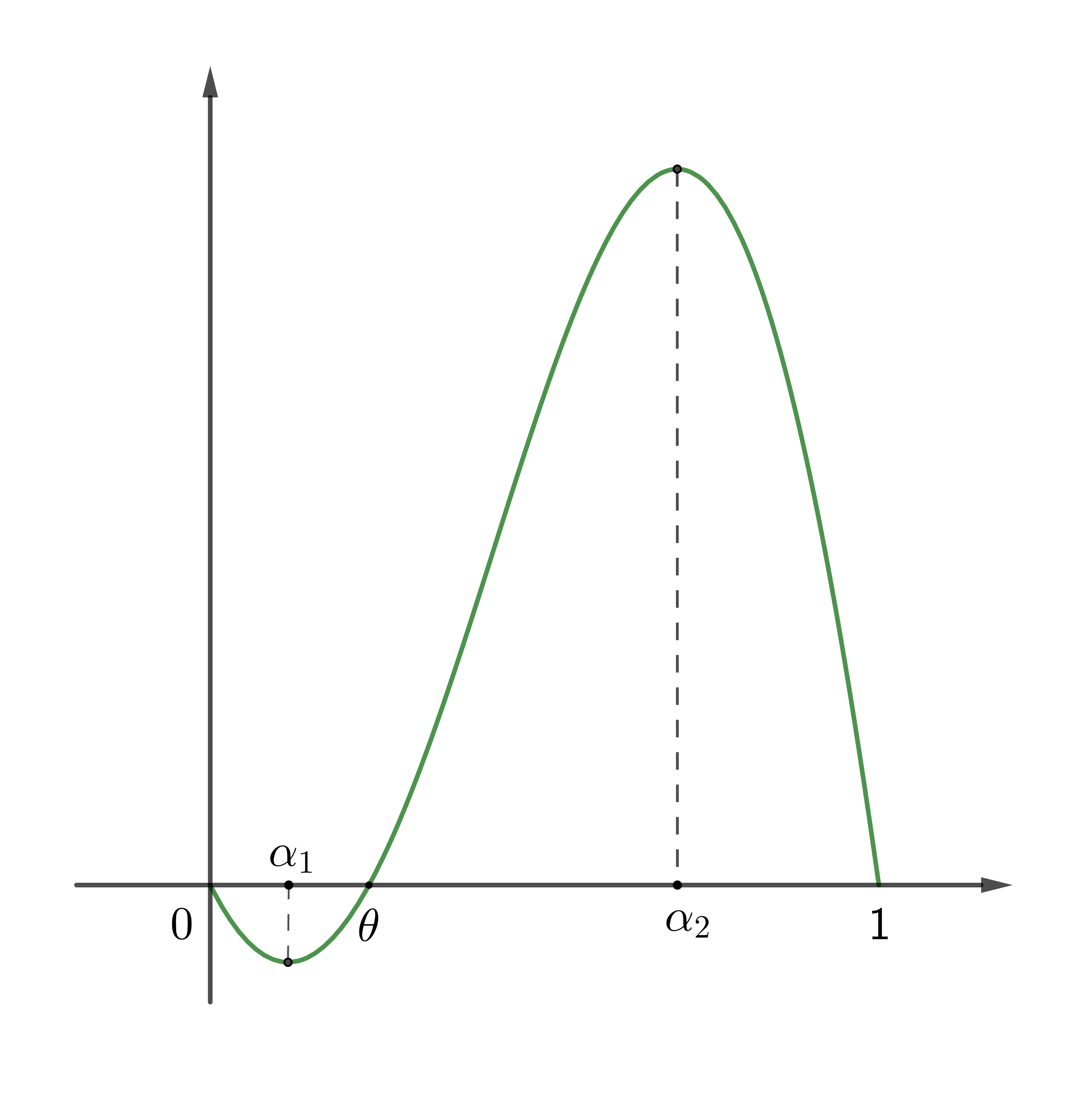} 
			\caption{$f(q)$}
			\setlength{\abovecaptionskip}{0pt}
			\setlength{\belowcaptionskip}{0pt}
			\label{fig:f(q)}
		\end{subfigure}
		\hfill
		\begin{subfigure}{0.45\textwidth}
			\centering
			\includegraphics[scale=3]{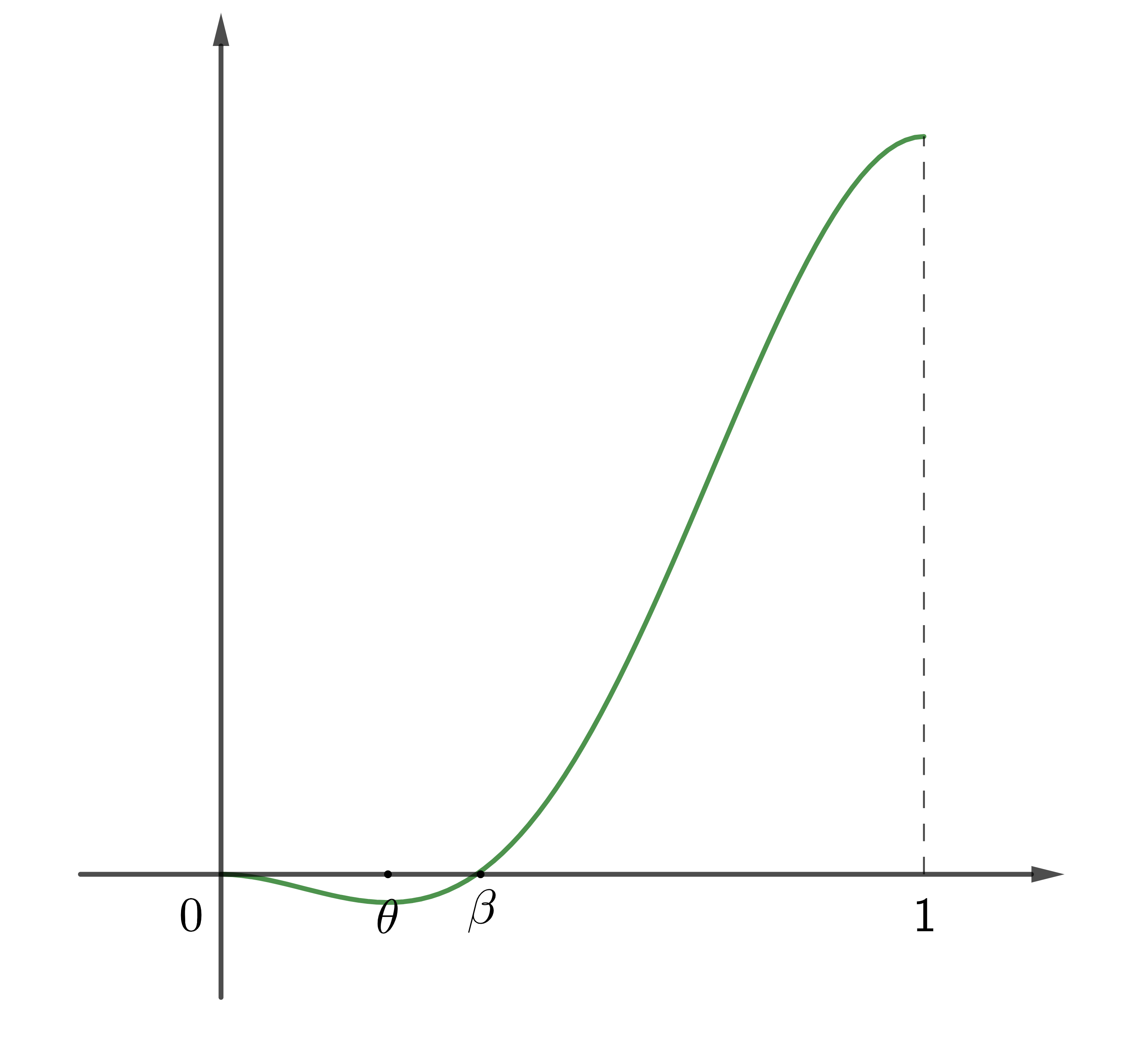} 
			\caption{$F(q)$}
			\setlength{\abovecaptionskip}{0pt}
			\setlength{\belowcaptionskip}{0pt}
			\label{fig:F(q)}
		\end{subfigure}
			\setlength{\abovecaptionskip}{0pt}
			\setlength{\belowcaptionskip}{0pt}
		\caption{Graph of functions $f$ and $F$}
		\label{fig:func}
	\end{figure} 
	The main results of the present work concerns existence and stability of steady-state solutions of \cref{eqn:pb1}, i.e. solutions of \cref{eqn:pb2}. 
	\subsection{Existence of steady-state solutions}
	\label{sec:exist}
	In our result, we first focus on two types of steady-state solutions defined as follows
	\begin{definition} Consider a steady-state solution $p(x)$,
		
		$p$ is called a symmetric-decreasing (SD) solution when $p$ is symmetric on $(-L,L)$ with values in $[0,1]$, decreasing on $(0,L)$ and $p'(0) = 0$ (see \cref{fig:p1}). 
		
		Similarly, $p$ is called a symmetric-increasing (SI) solution when $p$ is symmetric on $(-L,L)$ with values in $[0,1]$, increasing on $(0,L)$ and $p'(0) = 0$ (see \cref{fig:p2}). 
		
		Any solution which is either (SD) or (SI) is called a symmetric-monotone (SM) solution. 
	\end{definition}
	\begin{figure}
		\centering
		\begin{subfigure}{0.45\textwidth}
			\centering
			\includegraphics[width = \textwidth]{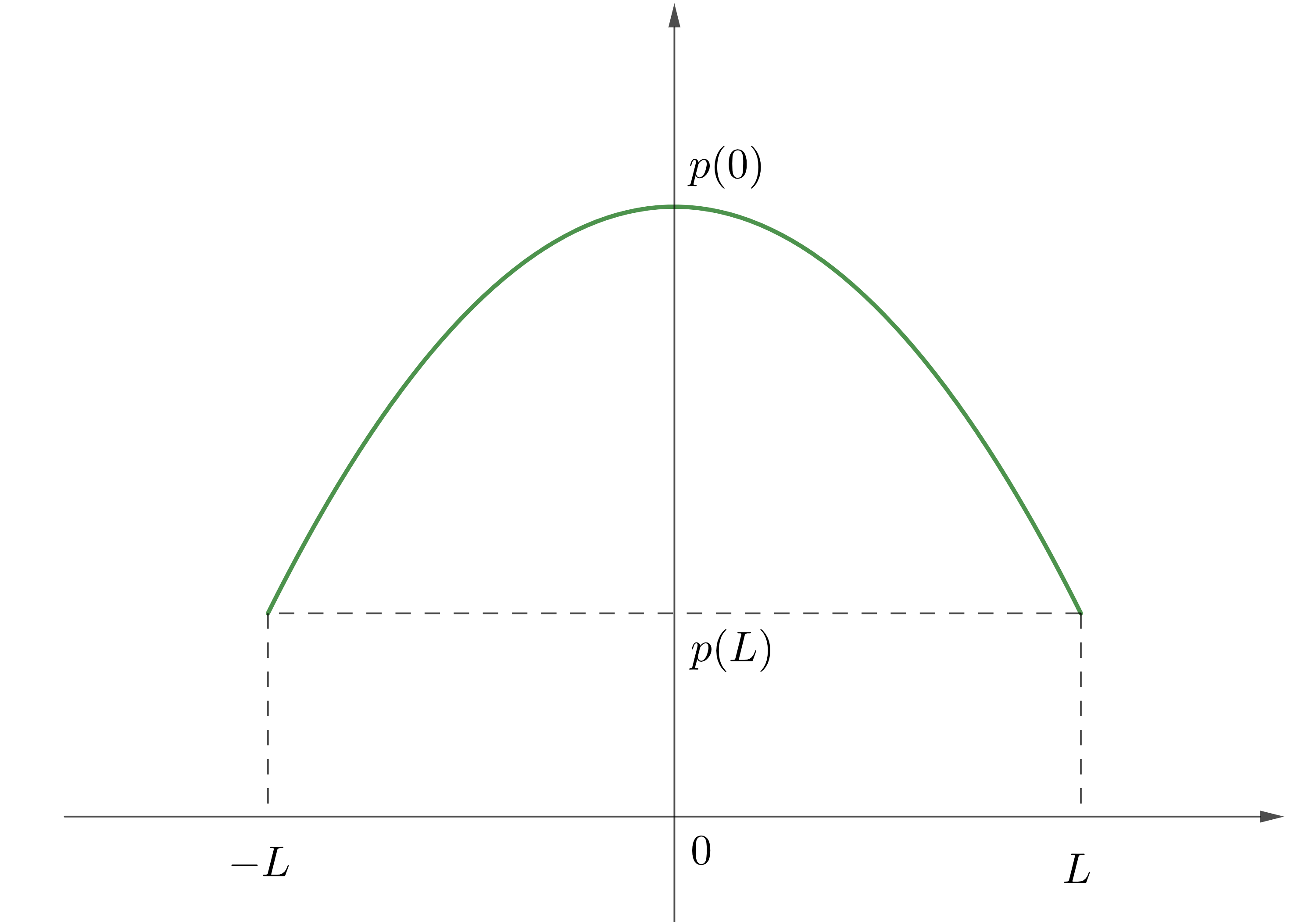} 
			\caption{(SD): $p$ decreasing on $(0,L)$}
			\label{fig:p1}
		\end{subfigure}
		\hfill
		\begin{subfigure}{0.45\textwidth}
			\centering
			\includegraphics[width = \textwidth]{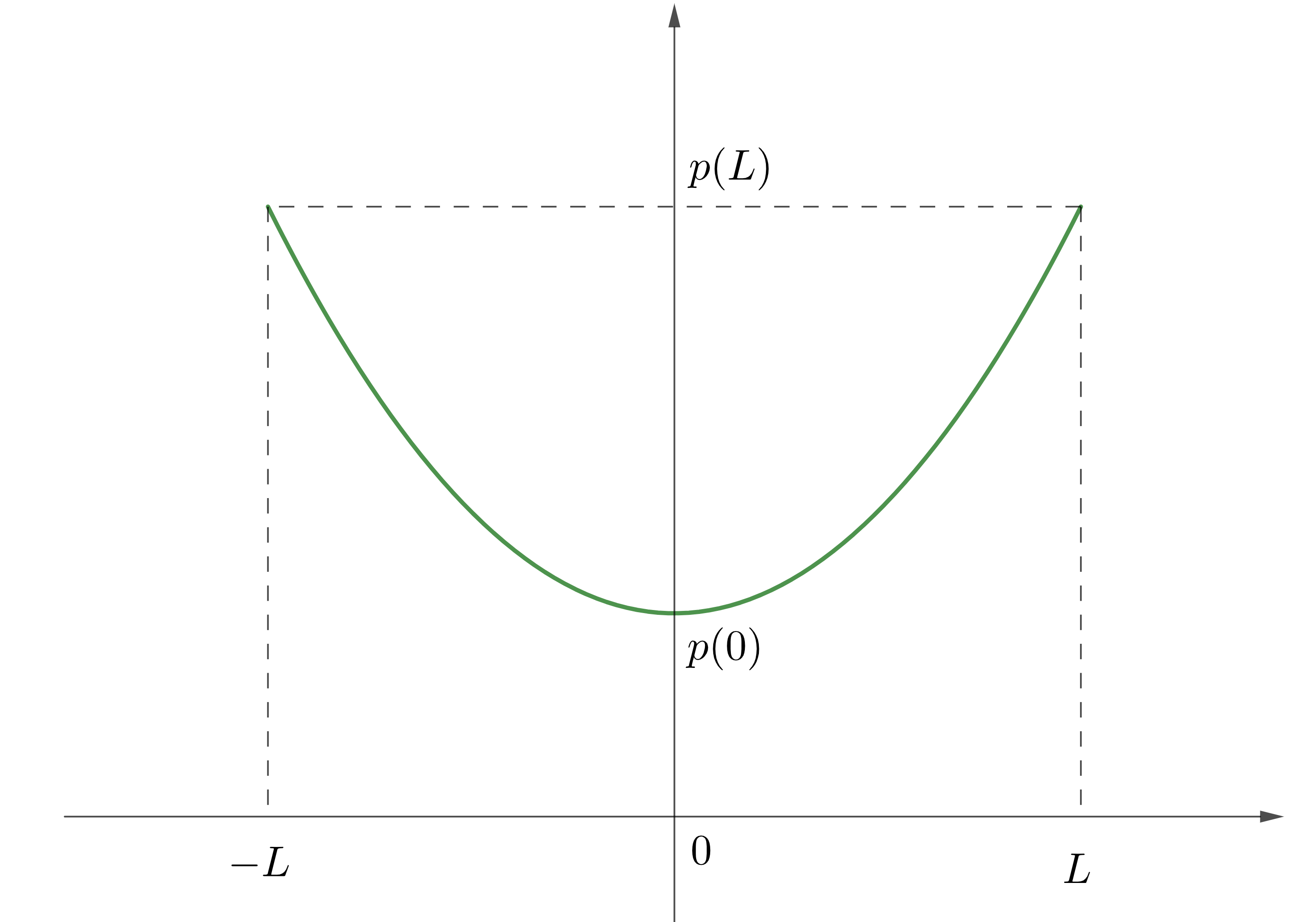} 
			\caption{(SI): $p$ increasing on $(0,L)$}
			\label{fig:p2}
		\end{subfigure}
			\setlength{\abovecaptionskip}{0pt}
			\setlength{\belowcaptionskip}{0pt}
		\caption{Symmetric steady-state solutions $p$}
		\label{fig:functionp}
	\end{figure} 

	The following theorems present the main result of existence of (SM) solutions depending upon the parameters. For each value of $p^\text{ext} \in (0,1)$ and $D > 0$, we find the critical values of $L$ such that \cref{eqn:pb2} admits solutions.
	\begin{theorem}
		\label{thm:exist}
		In a bounded domain $\Omega = (-L,L) \subset \mathbb{R}$, consider the stationary problem \cref{eqn:pb2}. Assume that the reaction term $f$ satisfies \cref{reaction} and \cref{convexity}. Then, there exist two functions
		\begin{equation}
		    \begin{array}{c r c l}
		     M_{d}, M_i: & (0,1) \times (0,+\infty) & \longrightarrow & [0,+\infty], \\
		     & (p^\text{ext},D) & \longmapsto & M_d(p^\text{ext},D), M_i(p^\text{ext},D),
		    \end{array}
		\end{equation}
		such that for any $p^\text{ext} \in (0,1), D > 0$, problem \cref{eqn:pb2} admits at least one (SD) solution (resp., (SI) solution) if and only if $L \geq M_{d}(p^\text{ext},D)$ (resp., $L \geq M_{i}(p^\text{ext},D)$) and the values of these solutions are in $[p^\text{ext}, 1]$ (resp., $[0,p^\text{ext}]$). More precisely,
		\begin{enumerate}[label=(\alph*),leftmargin=1\parindent]
			\item If $0 < p^\text{ext} < \theta$, then  for any $D > 0$, $M_{i}(p^\text{ext},D) = 0$ and $M_{d}(p^\text{ext},D) \in (0,+\infty)$.
			
			\noindent Moreover, if $p^\text{ext} \leq \alpha_1$, the (SI) solution is unique. 
			\item If $\theta < p^\text{ext} < 1$, then for any $D > 0$, $M_d(p^\text{ext},D) = 0$. If $\alpha_2 \leq p^\text{ext}$, the (SD) solution is unique. Moreover, consider $\beta$ as in \cref{eqn:beta},
			
			$\bullet$ if $p^\text{ext} \leq \beta$, then $M_i(p^\text{ext},D) \in (0,+\infty)$ for any $D > 0$;
			
			$\bullet$ if $p^\text{ext} > \beta$, then there exists a constant $D_* > 0$ such that $M_i(p^\text{ext},D) \in (0,+\infty)$ for any $D < D_*$, and $M_i(p^\text{ext}, D) = +\infty$ for $D \geq D_*$. 
			\item If $p^\text{ext} = \theta $, then $M_d(\theta, D) = M_i(\theta, D)= 0$. Moreover, there exists a constant solution $p \equiv \theta$.
		\end{enumerate}
	\end{theorem}
	In the statement of the above result, $M_i = 0$ means that for any $L >0$, \cref{eqn:pb2} always admits (SI) solutions. $M_i = +\infty$ means that there is no (SI) solution even when $L$ is large. The same interpretation applies for $M_d$. 
	
	Besides, problem \cref{eqn:pb2} can also admit solutions which are neither (SD) nor (SI). The following theorem provides an existence result for those solutions. 
	\begin{theorem}
	    \label{thm:nonSM}
	    In a bounded domain $\Omega = (-L,L) \subset \mathbb{R}$, consider the stationary problem \cref{eqn:pb2}. Assume that the reaction term $f$ satisfies \cref{reaction} and \cref{convexity}. Then, there exists a function 
	    	\begin{equation}
		    \begin{array}{c r c l}
		     M_*: & (0,1) \times (0,+\infty) & \longrightarrow & [0,+\infty], \\
		     & (p^\text{ext},D) & \longmapsto & M_*(p^\text{ext},D),
		    \end{array}
		\end{equation}
		such that for any $p^\text{ext} \in (0,1), D > 0$, problem \cref{eqn:pb2} admits at least one solution which is not (SM) if and only if $L \geq M_{*}(p^\text{ext},D)$. Moreover,
	    
	    \noindent $\bullet$ If $p^\text{ext} \leq \beta $, then for any $D > 0$, one has
	    \begin{equation}
	       0 <  M_i(p^\text{ext},D) + M_d(p^\text{ext},D) < M_*(p^\text{ext},D) < +\infty.
	    \end{equation}
	    $\bullet$ If $p^\text{ext} > \beta $, then for any $D < D_*$, one has $0 < M_i(p^\text{ext},D) < M_*(p^\text{ext},D) < +\infty$. Otherwise, for $D \geq D_*$, $M_*(p^\text{ext},D) = +\infty$. Here, $D_*$ was defined in \cref{thm:exist}.
	\end{theorem}
    \noindent The construction of $M_i, M_d, M_*$ will be done in the proof in \cref{sec:proof}. The idea of the proof is based on a careful study of the phase portrait of \cref{eqn:pb2}. 
    
	In the next section, we present a result about stability and instability of steady-state solutions of problem \cref{eqn:pb2}. 
	\subsection{Stability of steady-state solutions}
	The definition of stability and instability used in the present work comes from Lyapunov stability
	\begin{definition}
		A steady-state solution $p(x)$ of \cref{eqn:pb1} is called stable if for any constant $\epsilon > 0$, there exists a constant $\delta > 0$ such that when $||p^\text{init} - p||_{ \infty}  < \delta$, one has 
		\begin{equation}
			||p^0(t,\cdot) - p||_{ \infty}  < \epsilon, \quad \text{ for all } t > 0
		\end{equation}
	where $p^0(t,x)$ is the unique solution of \cref{eqn:pb1}. If, in addition, 
	\begin{equation}
		\displaystyle \lim_{t \rightarrow \infty}||p^0(t,\cdot) - p||_{ \infty} = 0,
	\end{equation}
	then $p$ is called asymptotically stable. The steady-state solution $p$ is called unstable if it is not stable. 
	\end{definition}
	
	The following theorem provides sufficient conditions for the stability of steady-state solutions given in \cref{sec:exist}.
	\begin{theorem}
	\label{thm:stability}
	In the bounded domain $\Omega = (-L,L) \subset \mathbb{R}$, consider the problem \cref{eqn:pb1} with the reaction term satisfying \cref{reaction} and \cref{convexity}. There exists a constant $\lambda_1 \in \left(0,\dfrac{\pi^2}{4L^2} \right)$ such that for any steady-state solution $p$ of \cref{eqn:pb1},
	
	$\bullet$ If $f'(p(x)) > \lambda_1$ for any $x \in (-L,L)$, then $p$ is unstable.
	
	$\bullet$ If $f'(p(x)) < \lambda_1$ for any $x \in (-L,L)$, then $p$ is asymptotically stable. 
	\end{theorem} 	
	
    The principle of linearized stability is used to prove this theorem (see \cref{sec:proof}). $\lambda_1$ is the principle eigenvalue of the linear problem and its value is the smallest positive solution of equation $\sqrt{\lambda}\tan{\left(L\sqrt{\lambda}\right)} = D$.
    \begin{remark}
        By \cref{convexity}, $f'(q) \leq  0 < \lambda_1$ for any $q \in [0,\alpha_1] \cup [\alpha_2,1] $, we can deduce that the steady-state solutions with values smaller than $\alpha_1$ or larger than $\alpha_2$ are asymptotically stable. 
    \end{remark}
    
    As a consequence of Theorems \ref{thm:exist}, \ref{thm:nonSM}, and \ref{thm:stability}, the following important result provides complete information about existence and stability of steady-state solutions in some special cases. 
    
    \begin{corollary}
    \label{result}
        In the bounded domain $\Omega = (-L,L) \subset \mathbb{R}$, consider the problem \cref{eqn:pb1} with the reaction term satisfying \cref{reaction} and \cref{convexity}. Then for any $D > 0$, we have
        
        $\bullet$ If $p^\text{ext} \leq \alpha_1$, for any $L > 0$, there exists exactly one (SI) steady-state solution $p$ and it is asymptotically stable. Moreover, if $L < M_d(p^\text{ext},D)$, then $p$ is the unique steady-state solution of \cref{eqn:pb1}.
        
        $\bullet$ If $p^\text{ext} \geq \alpha_2$, for any $L > 0$, there exists exactly one (SD) steady-state solution $p$ and it is asymptotically stable. Moreover, if $L < M_i(p^\text{ext},D)$, then $p$ is the unique steady-state solution of \cref{eqn:pb1}.
    \end{corollary}
    \begin{remark}
    This corollary gives us a comprehensive view about long-time behavior of solutions of \cref{eqn:pb1} when the size $L$ of the domain is small. In this case, the unique steady-state solution $p$ is symmetric, monotone on each half of $\Omega$ and asymptotically stable. Its values will be close to $0$ if $p^\text{ext}$ is small and close to $1$ if $p^\text{ext}$ is large. We discuss an essential application of this result in \cref{sec:bio}. 
    \end{remark}
\section{Proof of the theorems}
   \label{sec:proof}
   \subsection{Proof of existence}
   In this section, we use phase-plane analysis to prove the existence of both (SM) and non-(SM) steady-state solutions depending on the parameters. The studies of (SD) and (SI) solutions will be presented respectively in \cref{sec:SD} and \cref{sec:SI}. Then, using these results, we prove \cref{thm:exist}. The proof of \cref{thm:nonSM} will be presented after that using the same technique. 
   
   First, we introduce the following function
		\begin{equation}
			E(p,p') = \dfrac{(p')^2}{2} + F(p).
			\label{eqn:energy}
		\end{equation}
		Since $\dfrac{d}{dx} E(p,p') = p'(p'' + f(p)) = 0$, then $E(p,p')$ is constant along the orbit of \cref{eqn:pb2}. From \cref{extreme}(c), we can deduce that there exists an $x_0 \in (-L,L)$ such that $p'(x_0) = 0$, thus one has 
		\begin{equation}
			E(p(x_0),0) = E(p(x),p'(x)),
		\end{equation}
		for any $x \in (-L,L)$. Therefore, the relation between $p'$ and $p$ is as below
		\begin{equation}
			p' = \pm \sqrt{2F(p(x_0)) - 2F(p)}.
			\label{eqn:phase}
		\end{equation}
		
    \begin{figure}
			\centering
			\begin{subfigure}{0.45\textwidth}
				\centering
				\includegraphics[width = \textwidth]{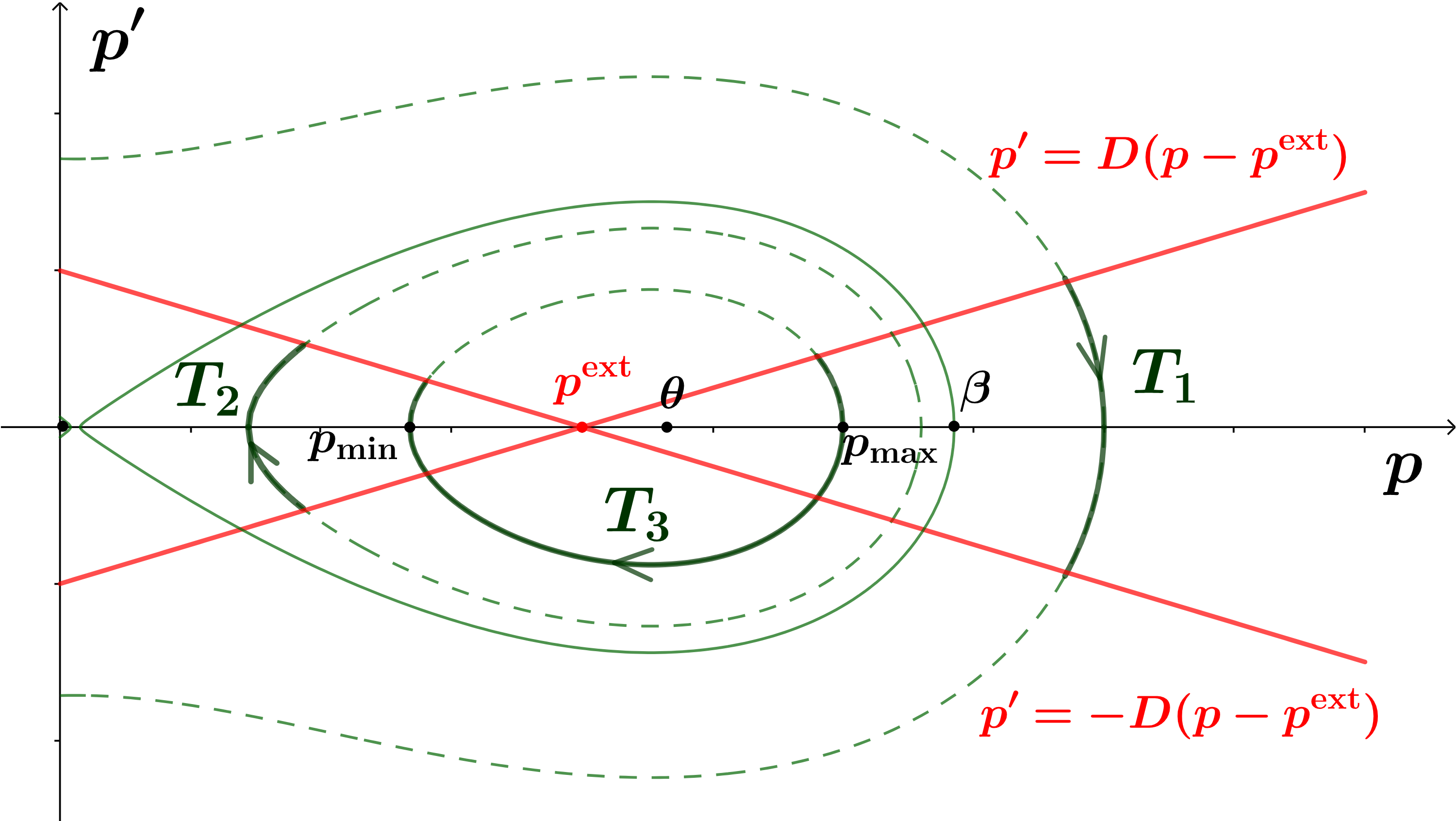} 
				\setlength{\abovecaptionskip}{0pt}
			    \setlength{\belowcaptionskip}{0pt}
				\caption{$p^\text{ext} < \theta < \beta, D > 0$}
				\label{fig:phase}
			\end{subfigure}
			\hfill
			\begin{subfigure}{0.45\textwidth}
				\centering
				\includegraphics[width = \textwidth]{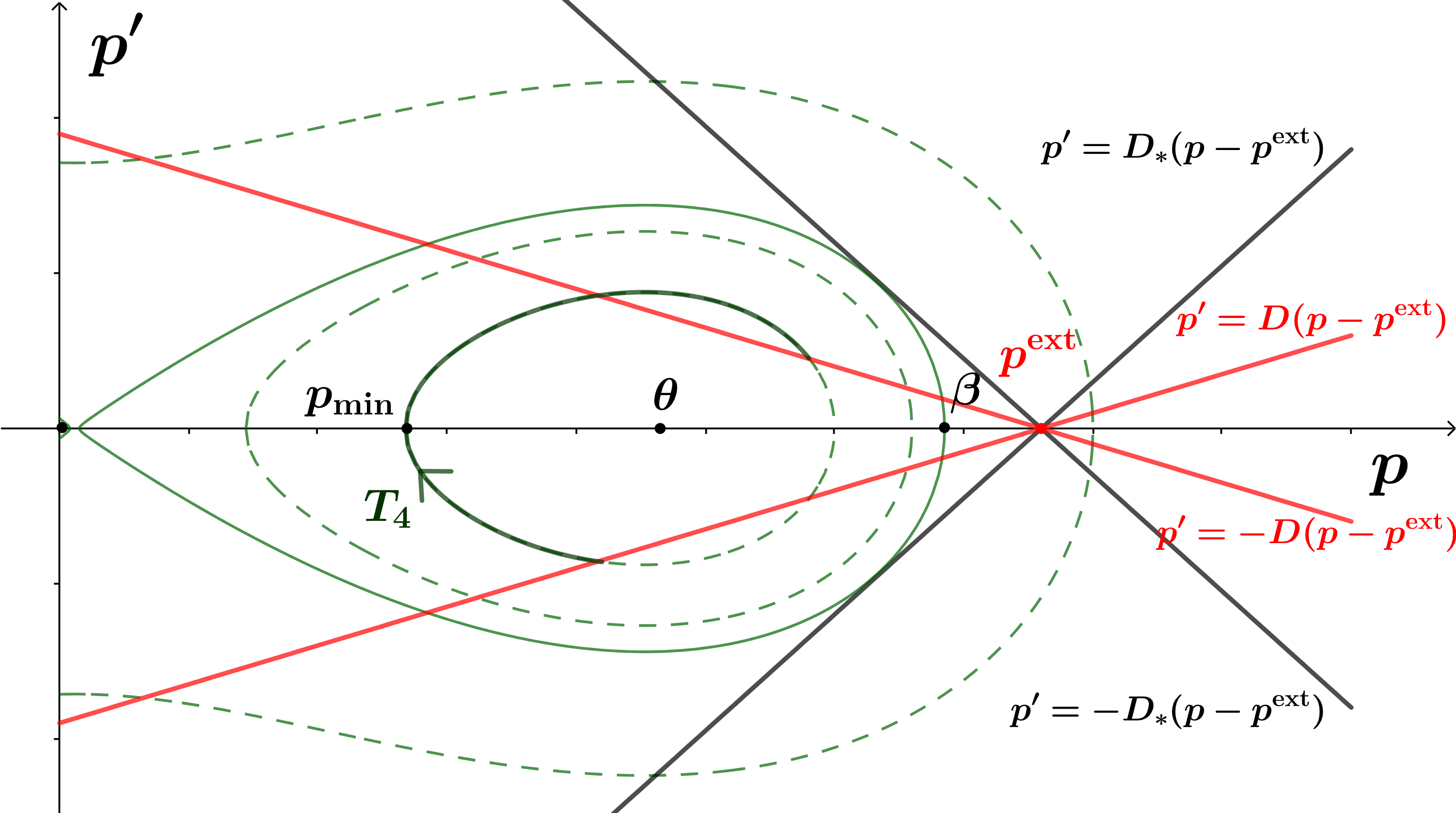} 
				\setlength{\abovecaptionskip}{0pt}
			    \setlength{\belowcaptionskip}{0pt}
				\caption{$p^\text{ext} > \beta, D \leq D_*$}
				\label{fig:phase2}
			\end{subfigure}
			\setlength{\abovecaptionskip}{5pt}
			\setlength{\belowcaptionskip}{0pt}
			\caption{Phase portrait of \cref{eqn:pb2}}
			\label{fig:phaseportrait}
	\end{figure} 

     According to this relation, one has a phase plane as in \cref{fig:phase}, in which the curves illustrate the relation between $p'(x)$ and $p(x)$ in \cref{eqn:phase} with respect to different values of $p(x_0)$. We can see that some curves do not end on the axis $p=0$ but wrap around the point $(\theta,0)$. This is dues to the fact that for any $p_1 \in [\theta,\beta]$, there exists a value $p_2 \in [0,\theta]$ such that $F(p_1) = F(p_2)$. Thus, if the curve passes through the point $(p_1,0)$, it will also pass through the point $(p_2,0)$ on the axis $p'=0$. Moreover, those curves only exist if their intersection with the axis $p'=0$ has $p$-coordinate less than or equal to $\beta$. Besides, the two straight lines show the relation between $p'$ and $p$ at the boundary points. Solutions of \cref{eqn:pb2} correspond to those orbits that connect the intersection of the curves with the line $p' = D(p-p^\text{ext})$ to the intersection of the curves with the line $p' = -D(p-p^\text{ext})$. 
		
    In the phase plane in \cref{fig:phase}, orbit $T_1$ describes a (SD) solution, while orbit $T_2$ corresponds to a (SI) solution. On the other hand,  the solid curve $T_3$ shows the orbit of a steady-state solution which is not symmetric-monotone.  
    \begin{remark}{\it (Graphical interpretation of $D_*$)} The (SI) solutions (see \cref{fig:p2}) have orbit as $T_2$ in \cref{fig:phase}. This type of orbits only exists when the lines $p = \pm D(p - p^\text{ext})$ intersect the curves wrapping around the point $(\theta,0)$. In the case when $p^\text{ext} > \beta$, the constant $D_* > 0$ in \cref{thm:exist} is the slope of the tangent line to the curve passing through $(\beta,0)$ as in \cref{fig:phase2}. Hence, if $D > D_*$, there exists no (SI) solution. We construct explicitly the value of $D_*$ in \cref{Type2} below.
    \end{remark}
   Next, we establish some relations between the solution $p$ and the parameters based on the phase portrait above. For any $x > x_0$, if $p$ is monotone on $(x_0,x)$, we can invert $x \mapsto p(x)$ into function $p \mapsto X(p)$. We obtain $X'(p) = \dfrac{\pm 1}{\sqrt{2F(p(x_0)) - 2F(p)}} $. By integrating this equation, we obtain that 
		\begin{equation}
				x - x_0 = \displaystyle \int_{p(x_0)}^{p(x)} \dfrac{ (-1)^k ds}{\sqrt{2F(p(x_0)) - 2F(s)}},
				\label{eqn:time}
		\end{equation}
	where $k = 1$ if $p$ is decreasing and $k = 2$ if $p$ is increasing on $(x_0,x)$. We can obtain the analogous formula for $x < x_0$.
	
	First, we focus on symmetric-monotone (SM) solutions for which $p'(0) = 0$, then we analyze the integral in \cref{eqn:time} with $x = L, x_0 = 0$. For any $p^\text{ext} \in (0,1)$, using \cref{eqn:phase}, we have
   \begin{equation}
   F(p(0)) = F(p(L)) + \dfrac{1}{2} D^2 \left( p(L) - p^\text{ext}\right)^2 = G(p(L)),
   \label{eqn:eq1}
   \end{equation}
   for $F$ defined in \cref{eqn:F} and 
   \begin{equation}
   G(q) := F(q) + \dfrac{1}{2} D^2 (q - p^\text{ext} )^2,
   \label{eqn:G}
   \end{equation}
   and from \cref{eqn:time} with $x=L, x_0 = 0$, we have
   \begin{equation}
   L = \displaystyle \int_{p(0)}^{p(L)} \dfrac{ (-1)^k ds}{\sqrt{2F(p(0)) - 2F(s)}},
   \label{eqn:eq2}
   \end{equation}
   where $k = 1$ if $p$ is decreasing on $(0,L)$, $k = 2$ if $p$ is increasing on $(0,L)$. 
   
   Thus, the (SM) solution of \cref{eqn:pb2} exists if there exist values $p(L)$ and $p(0)$ that satisfy \cref{eqn:eq1} and \cref{eqn:eq2}. When such values exist, we can assess the value of $p(x)$ for any $x$ in $(-L,L)$ using \cref{eqn:time}.
   
   Before proving existence of such values of $p(0)$ and $p(L)$, we establish some useful properties of the function $G$ defined in \cref{eqn:G}. It is continuous in $[0,1]$ and $G(q)\geq F(q) $ for all $q \in [0,1]$. Moreoever, the following lemma shows that $G$ has a unique minimum point.
   
   \begin{lemma}
   	\label{G}
   	For any $p^\text{ext} \in (0,1)$, there exists a unique value $\overline{q} \in (0,1)$ such that $G'(\overline{q}) = 0$, $G'(q) < 0$ for all $q \in [0,\overline{q})$ and $G'(q) > 0$ for all $q \in (\overline{q},1]$. Particularly, $G(\overline{q}) = \displaystyle \min_{[0,1]} G$.
   \end{lemma}
   \begin{proof}
   	We have $
   	G'(q) = f(q) + D^2(q-p^\text{ext})
   	$. 
   	We consider the following cases. 
   	
   	{\bf Case 1}: When $p^\text{ext} = \theta$, we have $G'(p^\text{ext}) = G'(\theta) = f(\theta) = 0, G'(q) < 0 $ for all $q \in (0,\theta)$ and $G'(q) > 0$ for all $q \in (\theta,1)$. Thus $\overline{q} = \theta = p^\text{ext}$. 
   	
   	{\bf Case 2}: When $p^\text{ext} < \theta$, we have $G'(q)  < 0$ for all $q \in [0,p^\text{ext}]$ and $G'(q) > 0$ for all $q \in [\theta,1]$. So there exists at least one value $\overline{q} \in (p^\text{ext},\theta)$ such that $G'(\overline{q}) = 0$. 
   	
   	For any $\overline{q} \in (p^\text{ext},\theta)$ such that $G'(\overline{q}) = 0$, we have $f(\overline{q}) + D^2 (\overline{q} - p^\text{ext}) = 0$ so that $D^2 = -\dfrac{f(\overline{q})}{\overline{q}- p^\text{ext}}$. We can prove that $G''(\overline{q})$ is strictly positive. Indeed, from \cref{convexity} we have that $\alpha_1$ is the unique value in $(0,\theta)$ such that $f'(\alpha_1) = 0$, thus $f(\alpha_1) = \displaystyle \min_{[0,\theta]} f < 0$. 
   	
   	If $\alpha_1 \leq \overline{q} < \theta $ then $f'(\overline{q}) \geq 0$. One has $	G''(\overline{q}) = f'(\overline{q}) + D^2 > 0$.
   	
   	If $p^\text{ext} < \overline{q} < \alpha_1$, due to the fact that $f$ is convex in $(0,\alpha_1)$ one has $f'(\overline{q}) \geq  \dfrac{f(\overline{q}) - f(p^\text{ext})}{\overline{q} - p^\text{ext}}$. Since $f(p^\text{ext}) < 0$, one has	$G''(\overline{q}) = f'(\overline{q}) + D^2 = f'(\overline{q}) - \dfrac{f(\overline{q})}{\overline{q} - p^\text{ext}} > f'(\overline{q}) + \dfrac{f(p^\text{ext}) -f(\overline{q})}{\overline{q} - p^\text{ext}} \geq 0.$	
   	One can deduce that $\overline{q}$ is the unique value in $(0,1)$ such that $G'(\overline{q}) = 0$ and $G(\overline{q}) = \displaystyle \min_{[0,1]} G$, so it satisfies \cref{G}.
   	
   	{\bf Case 3}: When $p^\text{ext} > \theta$, the proof is analogous to case 2 but using the concavity of $f$ in $(\alpha_2,1)$. We obtain that there exists a unique value $\overline{q}$ in $(\theta,p^\text{ext})$ which satisfies \cref{G}.
   \end{proof}
   
   When $p^\text{ext} = \theta$, it is easy to check that $p\equiv \theta$ is a solution of \cref{eqn:pb2}. We now analyze two types of (SM) solutions (see  \cref{fig:functionp}) in the following parts.
   \subsubsection{Existence of (SD) solutions}
   \label{sec:SD}
   In this part, the solution $p$ we study is symmetric on $(-L,L)$ and decreasing on $(0,L)$ (see \cref{fig:p1}). So $ p(L) < p(x) < p(0)$ for any $x \in (0,L)$. But from \cref{eqn:phase}, we have that $F(p(x)) \leq  F(p(0))$, so $F'(p(0)) \geq 0$. It implies that $p(0) \in [\theta,1]$. Next, we use two steps to study existence of (SD) solutions:
   
   \noindent {\bf Step 1: Rewriting as a non-linear equation on $p(L)$}
   
   For any $q \in (\theta,1)$, we have $F'(q) = f(q) > 0$ so $F|_{(\theta,1)}: (\theta, 1) \longrightarrow \left(F(\theta), F(1)\right)$ is invertible. Define $F_1^{-1} := (F|_{(\theta,1)})^{-1}: \left( F(\theta), F(1)\right) \longrightarrow (\theta,1)$, and $F_1^{-1}(F(\theta)) = \theta, F^{-1}_1(F(1)) = 1$. Then, $F^{-1}_1$ is continuous in $[F(\theta),F(1)]$. For any $y \in \left( F(\theta),F(1)\right)$, one has $\left( F^{-1}_1\right)'(y) = \dfrac{1}{F'\left( F^{-1}_1(y)\right)} = \dfrac{1}{f\left( F^{-1}_1(y)\right)} > 0 $, so $F^{-1}_1$ is an increasing function in $\left( F(\theta), F(1)\right)$. From \cref{eqn:eq1} and \cref{eqn:eq2}, since $p$ is decreasing in $(0,L)$, we have $L = \displaystyle \int_{p(L)}^{p(0)} \dfrac{ ds}{\sqrt{2G(p(L)) - 2F(s)}}$. Denote 
   \begin{equation}
   \mathcal{F}_1(q) := \displaystyle \int_{q}^{F_1^{-1}(G(q))}\dfrac{ds}{\sqrt{2G(q) - 2F(s)}}.
   \label{eqn:F1}
   \end{equation}
   Hence, a (SD) solution $p$ of system \cref{eqn:pb2} has $p(0) = F^{-1}_1(G(p(L)))$,	and $p(L)$ satisfies 
   \begin{equation}
   L = \mathcal{F}_1(p(L)).
   \label{eqn:sys1}
   \end{equation}	
   Moreover, one has $p'(x) \leq 0$ for all $x \in (0,L)$ thus $-D(p(L) - p^\text{ext}) = p'(L)  \leq 0$. One can deduce that 
   \begin{equation}
   p(L) \geq p^\text{ext}.
   \end{equation}
   
   \noindent {\bf Step 2: Solving \cref{eqn:sys1} in $[p^\text{ext},1]$}
   
   The existence of value $p(L)$ of the (SD) solutions is established as follows 
   \begin{proposition}
   	\label{Type1}
   	For any $D > 0, p^\text{ext} \in (0,1)$, we have 
   	\begin{enumerate}[leftmargin=1\parindent]
   		\item [1. ] If $0 < p^\text{ext} < \theta$, then there exists a constant $M_1 > 0$ such that equation \cref{eqn:sys1} has at least one solution $p(L) \geq p^\text{ext}$ if and only if $L \geq M_1$. 
   		\item [2. ] If $ \theta \leq p^\text{ext} < 1$, then equation \cref{eqn:sys1} admits at least one solution $p(L) \geq p^\text{ext}$ for all $L > 0$. If $p^\text{ext} \geq \alpha_2$, then this solution is unique. 
   	\end{enumerate} 
   \end{proposition}
   
   \begin{proof}
   	Since $F_1^{-1}$ is only defined in $[F(\theta),F(1)]$, we need to find $p(L) \in [p^\text{ext},1]$ such that $G(p(L)) \in [F(\theta),F(1)]$. 
   	
   	For all $q \in (0,1)$, we have $G(q) \geq F(q) \geq F(\theta)$ and from \cref{G}, there exists a value $\overline{q} \in (0,1)$ such that $\displaystyle \min_{[0,1]} G = G(\overline{q}) \leq G(p^\text{ext}) = F(p^\text{ext}) < \max_{[0,1]} F = F(1)$. Moreover, one has $G(1) > F(1)$, thus there exists a value $p^* \in (p^\text{ext},1)$ such that $G(p^*) = F(1)$. Then, for all $q \in [p^\text{ext},p^*], G(q) \in [F(\theta),F(1)]$ and we will find $p(L)$ in $[p^\text{ext},p^*]$. Since $F^{-1}_1$ increases in $(F(\theta),F(1))$, then $p(0) = F^{-1}_1(G(p(L))) \geq F^{-1}_1(F(p(L))) \geq p(L). $
   	
   	Function $\mathcal{F}_1$ in \cref{eqn:F1} is well-defined and continuous in $[p^\text{ext},p^*)$, $\mathcal{F} \geq 0$ in $[p^\text{ext},p^*)$. Moreover, since $F'(1) = 0$, one has $\displaystyle \lim_{p \rightarrow p^*}\mathcal{F}_1(p) =  \displaystyle \int_{p^*}^{1} \dfrac{ds}{\sqrt{2F(1) - 2F(s)}} = +\infty$.
   	
   	{\bf Case 1}: If $0 < p^\text{ext} < \theta$, we will prove that $\mathcal{F}_1$ is strictly positive in $[p^\text{ext},p^*)$. Indeed, for any $y \in [0,1]$, if $y < \theta$, by the definition of $F^{-1}_1$, we have $F^{-1}_1(G(y)) \in [\theta,1]$ so $F^{-1}_1(G(y)) > y$. If $y \geq \theta > p^\text{ext}$ then $G(y) = F(y) + \dfrac{1}{2} D^2(y - p^\text{ext})^2 > F(y)$ so again $F^{-1}_1(G(y)) > y$. Hence $\mathcal{F}_1(y) > 0$ for all $y \in [p^\text{ext},p^*)$. We have $\mathcal{F}_1(p) \rightarrow +\infty$ when $p \rightarrow p^*$, so there exists $p \in [p^\text{ext},p^*)$ such that $M_1:= \mathcal{F}_1(p) = \displaystyle \min_{[p^\text{ext},p^*]} \mathcal{F}_1 > 0$, and system \cref{eqn:sys1} admits at least one solution if and only if $L \geq M_1$. 
   	
   	{\bf Case 2}: If $\theta \leq p^\text{ext} < 1$, one has $G(p^\text{ext}) = F(p^\text{ext})$, then $F^{-1}_1(G(p^\text{ext})) = p^\text{ext}$ so 
   	$\mathcal{F}_1 (p^\text{ext}) = 0$. On the other hand, $\mathcal{F}_1(p) \rightarrow +\infty$ when $p \rightarrow p^*$. Thus, for any $L > 0$, there always exists at least one value $p(L) \in (p^\text{ext},p^*)$ such that $\mathcal{F}_1(p(L)) = L$. 
   	
   	Moreover, when $p^\text{ext} \geq \alpha_2$, we can prove that $\mathcal{F}_1' > 0$ on $(p^\text{ext},p^*)$. Indeed, denoting $\gamma(q) = F_1^{-1}(G(q))$, and changing the variable from $s$ to $t$ such that $s = t\gamma(q) + (1-t)q$, one has 
   	\begin{displaymath}
   	    \mathcal{F}_1(q) = \displaystyle \int_{0}^{1} \dfrac{[\gamma(q) - q]dt}{\sqrt{2F(\gamma(q)) - 2F(t\gamma(q) + (1-t)q)}}.
   	\end{displaymath}
   	To simplify, denote $s(q)= t\gamma(q) + (1-t)q$. For any $t \in (0,1)$, one has $q < s(q) < \gamma(q)$. Let us define $\Delta F = F(\gamma(q)) - F(s(q))$, then one has 
   	
   	\noindent $ \sqrt{2}\mathcal{F}_1'(q) = \displaystyle \int_0^1 (\gamma'(q)-1) (\Delta F)^{-1/2} dt - \dfrac{1}{2} \int_0^1 (\Delta F)^{-3/2}(\gamma(q)-q)\dfrac{d\Delta F}{dq} dt $
   	
   	$ = \displaystyle \int_0^1 (\Delta F)^{-3/2} \left[ (\gamma'(q)-1) \Delta F - \dfrac{1}{2} (\gamma(q)-q) (f(\gamma(q)) \gamma'(q) - f(s(q))s'(q)) \right]  $. 
   	
   	\noindent Let $P $ be the formula in the brackets, then
   	
   	\noindent $ \begin{array}{r l}
   	P &  =  (\gamma'-1) \Delta F - \dfrac{1}{2} (\gamma-q) \left[f(\gamma) \gamma' - f(s)(t\gamma' + 1 - t)\right]\\
   	& = (\gamma'-1) \left[\Delta F -\dfrac{1}{2}(\gamma-q)f(\gamma) + \dfrac{1}{2}(s-q)f(s)\right] - \dfrac{1}{2}(\gamma - q)(f(\gamma) - f(s)),
   	\end{array}
   	$
   	
   	Define $\psi(y) := F(y) - \dfrac{1}{2}f(y)(y - q)$ for any $y \in [q,\gamma(q)]$, then one has $\psi'(y) = \dfrac{1}{2}[f(y) - f'(y)(y-q)] \geq \dfrac{f(q)}{2}> 0$ since $y \geq q > p^\text{ext} \geq \alpha_2$ and $f$ is concave in $(\alpha_2,1)$, $f(q) > 0$. Moreover, $f$ is decreasing on $(\alpha_2,1)$ so $0 < f(\gamma(q)) < f(s(q)) < f(q)$, and $\gamma'(q) = \dfrac{G'(q)}{f(F_1^{-1}(G(q)))} = \dfrac{f(q) + D^2 (q - p^\text{ext})}{f(\gamma(q))} > 1$. Hence, we can deduce that $P = (\gamma'-1)(\psi(\gamma) - \psi(s)) - \dfrac{1}{2}(\gamma - q)(f(\gamma) - f(s)) > 0$ for any $t \in (0,1)$.  This proves that function $\mathcal{F}_1$ is increasing on $(p^\text{ext},p^*)$, so the solution of equation \cref{eqn:sys1} is unique. 
   \end{proof}
   \subsubsection{Existence of (SI) solutions} 
   \label{sec:SI}
   In this case, the technique we use to prove existence of (SI) solutions is analogous to (SD) solutions except the case when $p^\text{ext} > \beta$ (case 3 below). Since the proof is not straight forward, it is worth to re-establish this technique for (SI) solutions in two following steps:
   
   \noindent {\bf Step 1: Rewriting as a non-linear equation on $p(L)$}
   
   Since now $p$ is symmetric on $(-L,L)$ and increasing in $(0,L)$ (see \cref{fig:p2}), then $ p(0) < p(x) < p(L)$ for any $x \in (0,L)$. But from \cref{eqn:phase}, we have that $F(p(x)) \leq  F(p(0))$, so $F'(p(0)) \leq 0$. This implies that $p(0) \in [0,\theta]$.
   
   For any $q \in (0,\theta)$, we have $F'(q) = f(q) < 0$ so $F|_{(0,\theta)}: (0,\theta) \longrightarrow \left(F(\theta), F(0)\right)$ is invertible. Define $F_2^{-1} := (F|_{(0,\theta)})^{-1}: \left( F(\theta), F(0)\right) \longrightarrow (0,\theta)$, $F_2^{-1}(F(\theta)) = \theta, F^{-1}_2(F(0)) = 0$, and $F^{-1}_2$ is continuous in $[F(\theta),F(0)]$. For any $y \in \left( F(\theta), F(0)\right)$, $\left( F^{-1}_2\right)'(y) = \dfrac{1}{F'\left( F^{-1}_2(y)\right)} = \dfrac{1}{f\left( F^{-1}_2(y)\right)} < 0$, so $F^{-1}_2$ is a decreasing function in $\left( F(\theta), F(0)\right)$. From \cref{eqn:eq1} and \cref{eqn:eq2}, we have $L = \displaystyle \int_{p(0)}^{p(L)} \dfrac{ds}{\sqrt{2G(p(L)) - 2F(s)}}$. Denote
   \begin{equation}
   \mathcal{F}_2(q) := \displaystyle \int_{F^{-1}_2(G(q))}^{q} \dfrac{ds}{\sqrt{2G(q) - 2F(s)}}.
   \label{eqn:F2}
   \end{equation}
   Hence, a (SI) solution of system \cref{eqn:pb2} has $p(0) = F^{-1}_2(G(p(L)))$, and $p(L)$ satisfies 
   \begin{equation}
   L = \mathcal{F}_2(p(L)),
   \label{eqn:sys2}
   \end{equation}
   and in this case, one needs to find $p(L)$ in $[0,p^\text{ext}]$. 
   
   \noindent {\bf Step 2: Solving of \cref{eqn:sys2} in $[0,p^\text{ext}]$}
   \begin{proposition}
   	\label{Type2}
   	For any $p^\text{ext} \in (0,1)$, considering the value $\beta$ as in \cref{eqn:beta}, we have:  
   	\begin{enumerate}[leftmargin=1\parindent]
   		\item [1. ] If $ 0 < p^\text{ext} \leq \theta$, then equation \cref{eqn:sys2} admits at least one solution $p$ with $p(L) \leq p^\text{ext}$ for all $L > 0, D > 0$. If $p^\text{ext} \leq \alpha_1$, this solution is unique.
   		\item [2. ] If $\theta < p^\text{ext} \leq  \beta$, then for all $D > 0$, there exists a constant $M_2 > 0$ such that equation \cref{eqn:sys2} has at least one solution $p$ with $p(L) \leq p^\text{ext}$ if and only if $L \geq M_2$. 
   		\item [3. ] If $\beta < p^\text{ext} < 1$, then there exists a constant $D_* > 0$ such that when $D \geq D_*$, equation \cref{eqn:sys2} has no solution. Otherwise, there exists a constant $M_3 > 0$ such that equation \cref{eqn:sys2} has at least one solution $p$ with $p(L) \leq p^\text{ext}$ if and only if $L \geq M_3 $. 
   	\end{enumerate} 
   \end{proposition}
   \begin{proof}
   	As we assume that $F(0) < F(1)$ and $F(\theta) < F(0)$ then, due to the continuity of $F$, one can deduce that there exists a value $\beta \in (\theta,1)$ such that $F(\beta) = F(0) = 0$.
   	
   	Since $F_2^{-1}$ is only defined in $[F(\theta),F(0)]$, we need to find $p(L) \in [0,p^\text{ext}]$ such that $G(p(L)) \in [F(\theta),F(0)]$. For all $q \in (0,1)$, we have $G(q) \geq F(q) \geq F(\theta)$, thus equation \cref{eqn:sys2} has solutions if and only if  $\displaystyle \min_{[0,1]} G < F(0)$. Even when $\displaystyle \min_{[0,1]} G = G(\overline{q}) = F(0)$, $\mathcal{F}_2$ is still not defined in $[0,1]$ since $\mathcal{F}_2 (\overline{q}) = +\infty$.
   	
   	One has the following cases:
   	
   	{\bf Case 1:} $0 < p^\text{ext} \leq \theta$: 
   	
   	We have $\displaystyle \min_{[0,1]} G = G(\overline{q}) \leq G(p^\text{ext}) = F(p^\text{ext}) < \max_{[0,\theta]}F =  F(0)$, and $G(0) > F(0)$ so there is a value $p_* \in (0,p^\text{ext})$ such that $G(p_*) = F(0)$. Moreover $F'(0) = 0$, then $\displaystyle \lim_{p \rightarrow p^*}\mathcal{F}_2(p) = +\infty$. Thus, function $\mathcal{F}_2$ is only well-defined and continuous in $(p_*,p^\text{ext}]$. 
   	
   	When $0 < p^\text{ext} \leq \theta$, $F^{-1}_2(G(p^\text{ext})) = F^{-1}_2(F(p^\text{ext})) = p^\text{ext}$ so $\mathcal{F}_2 (p^\text{ext}) = 0$. We can deduce that for any $L > 0$, there always exists at least one value $p(L) \in (p_*,p^\text{ext})$ such that $\mathcal{F}_2(p(L)) = L$. 
   	When $p^\text{ext} \leq \alpha_1$, arguing analogously to the second case of \cref{Type1}, one has $\mathcal{F}_2' < 0$ on $(p_*,p^\text{ext})$, thus the solution is unique. 
   	
   	{\bf Case 2:} $\theta < p^\text{ext} \leq \beta$: 
   	
   	Since $F$ increases on $(\theta,1)$, then $\displaystyle \min_{[0,1]} G = G(\overline{q}) < G(p^\text{ext}) = F(p^\text{ext}) \leq F(\beta) =  F(0)$. Analogously to the previous case, $\mathcal{F}_2$ is well-defined and continuous in $(p_*,p^\text{ext}] $, $\displaystyle \lim_{p \rightarrow p^*}\mathcal{F}_2(p) = +\infty$, and $\mathcal{F}_2$ is strictly positive in $(p_*,p^\text{ext}]$. Therefore, there exists $p \in (p_*,p^\text{ext}]$ such that 
   	\begin{equation}
   	M_2:= \mathcal{F}_2(p) = \min_{[p_*,p^\text{ext}]} \mathcal{F}_2 > 0,
   	\end{equation}
   	and system \cref{eqn:sys2} admits as least one solution if and only if $L \geq M_2$. 
   	
   	{\bf Case 3:} $\beta < p^\text{ext} < 1$: 
   	
   	Consider the function $H(q) = F(q) + \dfrac{1}{2} f(q)(p^\text{ext} - q)$ defined in an interval $[\theta,p^\text{ext}]$. For any $\theta < q < p^\text{ext}$, one can prove that $H'(q) \geq 0$. 
   	
   	Indeed, if $q \leq \alpha_2$, then $f'(q) \geq 0$, and $f(q) >0$. One has $H'(q) = \dfrac{1}{2} f(q) + \dfrac{1}{2}f'(q)(p^\text{ext} - q)  > 0$. If $q > \alpha_2$, from \cref{convexity}, the function $f$ is concave in $(\alpha_2,1)$, and hence $f'(q)(p^\text{ext} - q) \geq f(p^\text{ext}) - f(q)$. Thus,
   	
   	\begin{center}
   		$ H'(q) = \dfrac{1}{2} (p^\text{ext} - q) \left(f'(q) + \dfrac{f(q)}{p^\text{ext} - q}\right) > \dfrac{1}{2} (p^\text{ext} - q) \left(f'(q) + \dfrac{f(q) - f(p^\text{ext})}{p^\text{ext} - q}\right) \geq 0.$
   	\end{center}
   	
   	Therefore, function $H$ increases in $(\theta,p^\text{ext})$. Moreover $H(\theta) = F(\theta) < F(0)$ and $H(p^\text{ext}) = F(p^\text{ext}) > F(\beta) = F(0)$, and so there exists a unique value $\overline{p}_* \in (\theta,p^\text{ext})$ such that $H(\overline{p}_*) = F(0)$. Take $D_* > 0$ such that $D_*^2 = \dfrac{f(\overline{p}_*)}{p^\text{ext} - \overline{p}_*}$. Then, for any $D > 0$,  from \cref{G}, there is a unique value $\overline{q} \in (\theta,p^\text{ext})$ such that $G'(\overline{q}) = 0$, $G(\overline{q}) = \displaystyle \min_{[0,1]}G$, and $D^2 = 
   	\dfrac{f(\overline{q})}{p^\text{ext} - \overline{q}}$. If $D < D_*$, then $\dfrac{f(\overline{q})}{p^\text{ext} - \overline{q}} < \dfrac{f(\overline{p}_*)}{p^\text{ext} - \overline{p}_*}$. 
   	
   	 Let $h(q) = \dfrac{f(q)}{p^\text{ext} - q}$, then $h'(q) = \dfrac{1}{p^\text{ext} - q} \left( f'(q) + \dfrac{f(q)}{p^\text{ext} - q} \right) > 0$ for $q \in (\theta, p^\text{ext})$. So function $h$ is increasing in $(\theta, p^\text{ext})$, and we can deduce that $\overline{q} < \overline{p}_*$. Hence, $\displaystyle \min_{[0,1]}G = G(\overline{q}) = F(\overline{q}) + \dfrac{1}{2}D^2(p^\text{ext} - \overline{q})^2 = F(\overline{q}) + \dfrac{1}{2}f(\overline{q})(p^\text{ext}- \overline{q}) = H(\overline{q}) < H(\overline{p}_*) = F(0)$. 
   	
   	Moreover, $G(p^\text{ext}) = F(p^\text{ext}) > F(\beta ) = F(0)$, $G(0) > F(0)$. Thus, there exists a maximal interval $(q_*,q^*) \subset [0,p^\text{ext}]$ such that $G(q) \in (F(\theta),F(0))$ for all $q \in (q_*,q^*)$. We have $0 < q_* < \overline{q} < q^* < p^\text{ext}$ and $G(q_*) = G(q^*) = F(0)$. Therefore, $\mathcal{F}_2$ is well-defined and continuous in $(q_*,q^*)$, and $	\displaystyle \lim_{p \rightarrow q^*}\mathcal{F}_2(p) = \lim_{p \rightarrow q_*} \mathcal{F}_2(p) = +\infty$. Reasoning like in the previous case, \cref{eqn:sys2} admits solution if and only if $L \geq M_3$, where
   	\begin{equation}
   	M_3 := \displaystyle \min_{[q_*,q^*]} \mathcal{F}_2 > 0,
   	\end{equation}
   	On the other hand, if $D \geq D_*$, $\displaystyle \min_{[0,1]} G \geq F(0)$, and equation \cref{eqn:sys2} has no solution. 
   \end{proof}
   \begin{proof}[Proof of \cref{thm:exist}] 
       As we showed in \cref{sec:SD}, the (SD) steady-state solution $p$ of \cref{eqn:pb2} has $p(L)$ satisfying equation \cref{eqn:sys1}. From \cref{Type1}, we can deduce that for fixed $p^\text{ext} \in (0,1), D > 0$, $M_d(p^\text{ext}, D) = \displaystyle \min_q \mathcal{F}_1(q)$. Thus, we obtain the results for (SD) steady-state solutions of \cref{eqn:pb2} in \cref{thm:exist}. 
       
       Similarly, \cref{Type2} provides that for fixed $p^\text{ext} \in (0,1), D > 0$, we have $M_i(p^\text{ext}, D) = \displaystyle \min_q \mathcal{F}_2(q)$ when $p^\text{ext} \leq \beta$ or $D < D_*$. Otherwise, $M_i(p^\text{ext}, D) = +\infty$. 
   \end{proof}
   \subsubsection{Existence of non-(SM) solutions}   As we can see in the phase portrait in \cref{fig:phaseportrait}, there exist some solutions of \cref{eqn:pb2} which are neither (SD) nor (SI). These solutions can be non-symmetric or can have more than one (local) extremum. By studying these cases, we prove \cref{thm:nonSM} as follows
   \begin{proof}[Proof of \cref{thm:nonSM}]
       We can see from \cref{fig:phase} that for fixed $p^\text{ext} \leq \beta, D > 0$, the non-(SM) solutions $p$ of \cref{eqn:pb2} have more than one (local) extreme value because their orbits have at least two intersections with the axis $p' = 0$  (see e.g. $T_3$). Those solutions have the same local minimum values, denoted $p_\text{min}$, and the same maximum values, denoted $p_\text{max}$. Moreover, we have $p_\text{min} < \theta < p_\text{max}$, and $F(p_\text{min}) = F(p_\text{max})$. 
       
       Since the orbits make a round trip of distance $2L$, then the more extreme values a solution has, the larger $L$ is. Hence, to find the minimal value $M_*$, we study the case when $p$ has one local minimum and one local maximum with orbit as $T_3$ in \cref{fig:phase}. Then we have 
       \begin{center}
           $G(p(-L)) = G(p(L)) = F(p_\text{min}) = F(p_\text{max})$,
       \end{center}
       and by using \cref{eqn:time}, we obtain 
     \begin{center}
         $2L = \mathcal{F}_1((p(-L)) + \displaystyle \int_{p_\text{min}}^{p_\text{max}} \dfrac{ds}{\sqrt{2F(p_\text{min})- 2F(s)}} + \mathcal{F}_2(p(L))$.
     \end{center}
       Using the same idea as above, we can show that $L$ depends continuously on $p(-L)$. 
       
       Moreover, $\displaystyle \int_{p_\text{min}}^{p_\text{max}} \dfrac{ds}{\sqrt{2F(p_\text{min})- 2F(s)}} > \mathcal{F}_1(p(-L)) + \mathcal{F}_2(p(L))$, and $M_d = \min \mathcal{F}_1,$ $M_i = \min \mathcal{F}_2$, therefore there exists a constant $M_* > M_d + M_i$ such that \cref{eqn:pb2} admits at least one non-(SM) solution $p$ if and only if $L \geq M_*$. 
       
       On the other hand, for fixed $p^\text{ext} > \beta, D < D_*$, it is possible that \cref{eqn:pb2} admits a non-symmetric solution with only one minimum. The orbit of this solution is as $T_4$ in \cref{fig:phase2}. In this case, we have $G(p(L)) = G(p(-L)) = F(p_\text{min})$ with $p(-L) < p(L)$ and
       \begin{center}
            $2L = \mathcal{F}_2(p(-L)) + \mathcal{F}_2(p(L)) > 2M_i$.
       \end{center}
       Hence, in this case we only need $M_* > M_i$.  
   \end{proof}
   
   \subsection{Stability analysis}
   We first study the principal eigenvalue and eigenfunction for the linear problem. Then by using these eigenelements, we construct the super- and sub-solution of \cref{eqn:pb1} and prove the stability and instability corresponding to each case in \cref{thm:stability}.
   \begin{proof}[Proof of \cref{thm:stability}]
       Consider the corresponding linear eigenvalue problem: 
   \begin{equation}
        \begin{cases}
            \begin{array}{r l}
                -\phi''  & = \lambda \phi \qquad  \text{ in } (-L,L), \\
                \phi'(L) & = -D\phi(L)  ,\\
                \phi'(-L)& = D\phi(-L), 
            \end{array}
        \end{cases}
        \label{eqn:EVP}
   \end{equation}
   where $\lambda$ is an eigenvalue  with associated eigenfunction $\phi$. We can see that $\phi = \cos{\left(\sqrt{\lambda}x\right)}$ is an eigenfunction iff $\sqrt{\lambda}\tan{\left(L\sqrt{\lambda}\right)} = D$. Denote $\lambda_1$ the smallest positive value of $\lambda$ which satisfies this equality, thus $L\sqrt{\lambda_1} \in \left( 0, \dfrac{\pi}{2}\right)$. Hence, $\lambda_1 \in \left(0, \dfrac{\pi^2}{4L^2} \right)$. Moreover, for any $x \in (-L,L)$, the corresponding eigenfunction $\phi_1(x) = \cos{\left( \sqrt{\lambda_1}x \right)}$ takes values in $(0,1)$.  
   
   {\bf Proof of stability:}
   Now let $p$ be a steady-state solution of \cref{eqn:pb1} governed by \cref{eqn:pb2}. First, we prove that if $f'(p(x)) < \lambda_1$ for any $x \in (-L,L)$ then $p$ is asymptotically stable. Indeed, since $f'(p(x)) < \lambda_1$, there exist positive constants $\delta, \gamma$  with $\gamma < \lambda_1$ such that for any $\eta \in [0,\delta]$,
   \begin{equation}
       f(p+\eta) - f(p) \leq (\lambda_1 - \gamma) \eta, \qquad  f(p) - f(p-\eta) \leq (\lambda_1 - \gamma) \eta,
       \label{eqn:lambda1}
   \end{equation}
   on $(-L,L)$. Now consider 
   \begin{displaymath}
       \overline{p}(t,x) = p(x) + \delta e^{-\gamma t} \phi_1(x), \qquad \underline{p}(t,x) = p(x) - \delta e^{-\gamma t} \phi_1(x). 
   \end{displaymath}
   Assume that $p^\text{init}(x) \leq p(x) + \delta \phi_1(x)$. Then by \cref{eqn:lambda1}, we have that $\overline{p}$ is a super-solution of \cref{eqn:pb1} because
   \begin{displaymath}
       \partial_t\overline{p} - \partial_{xx}\overline{p} = (\lambda_1 - \gamma) \delta e^{-\gamma t} \phi_1(x) + f(p) \geq f(p + \delta e^{-\gamma t} \phi_1(x)) = f(\overline{p}),
   \end{displaymath}
   due to the fact that $ 0 < \delta e^{-\gamma t} \phi_1(x) < \delta $ for any $t > 0$, $x \in (-L,L)$. Moreover, at the boundary points one has $\frac{\partial \overline{p}}{\partial \nu} + D(\overline{p} - p^\text{ext}) = \frac{\partial p}{\partial \nu} + D(p - p^\text{ext}) = 0.$
   
   Similarly, if we have $p^\text{init}(x) \geq p(x) - \delta \phi_1(x)$, and so $\underline{p}$ is a sub-solution of \cref{eqn:pb1}. Then, by the method of super- and sub-solution (see e.g. \cite{PAO}), the solution $p^0$ of \cref{eqn:pb1} satisfies $\underline{p} \leq p^0 \leq \overline{p}$. Hence, $ |p^0(t,x) - p(x)| \leq \delta e^{-\gamma t}\phi_1(x)$. Therefore, we can conclude that, whenever $|p^\text{init}(x) - p(x)| \leq \delta \phi_1(x)$ for any $x \in (-L,L)$, the solution $p^0$ of \cref{eqn:pb1} converges to the steady-state $p$ when $t \rightarrow +\infty$. This shows the stability of $p$. 
    
    {\bf Proof of instability: } In the case when $f'(p(x)) > \lambda_1$, there exist positive constants $\delta, \gamma$, with $\gamma < \lambda_1$, such that for any $\eta \in [0,\delta]$,
    \begin{equation}
         f(p+\eta) - f(p) \geq (\lambda_1 + \gamma) \eta,
         \label{eqn:lambda2}
    \end{equation}
    on $(-L,L)$. 
    
    For any $p^\text{init} > p$, there exists a positive constant $\sigma < 1$ such that $p^\text{init} \geq p + \delta (1- \sigma)$. Then $\widetilde{p}(t,x) = p(x) + \delta (1 - \sigma e^{-\gamma' t}) \phi_1(x)$, with $\gamma' < \gamma$ small enough, is a sub-solution of \cref{eqn:pb1}. Indeed, by applying \cref{eqn:lambda2} with $\eta =  \delta (1 - \sigma e^{-\gamma' t}) \phi_1(x) \in [0,\delta]$ for any $x \in (-L,L)$, we have 
    \begin{displaymath}
         \partial_t \widetilde{p} - \partial_{xx} \widetilde{p} = \gamma' \delta \sigma e^{-\gamma' t} \phi_1(x) + \lambda_1 \delta (1 - \sigma e^{-\gamma' t}) \phi_1(x) + f(p) \leq f(p + \delta (1 - \sigma e^{-\gamma' t}) \phi_1(x))
    \end{displaymath}
    if $\gamma \geq \dfrac{\gamma'\sigma e^{-\gamma' t}}{1 - \sigma e^{-\gamma' t}} = \dfrac{\gamma' \sigma}{e^{\gamma' t} - \sigma} $ for any $t \geq 0$. This inequality holds when we choose $\gamma' \leq \dfrac{\gamma (1-\sigma)}{\sigma}$. Now, we have that $\widetilde{p}$ is a sub-solution of \cref{eqn:pb1}, thus for any $t \geq 0, x \in (-L,L)$, the corresponding solution $p^0$ satisfies 
    \begin{displaymath}
         p^0(t,x) - p(x) \geq \tilde p(t,x)-p(x) \geq \delta (1 - \sigma e^{-\gamma' t}) \phi_1(x). 
    \end{displaymath}
    Hence, for a given positive $\epsilon < \delta \displaystyle \min_{x} \phi_1(x)$, when $t \rightarrow +\infty$, solution $p^0$ cannot remain in the $\epsilon$-neighborhood of $p$ even if $p^\text{init} - p$ is small. This implies the instability of $p$. 
   \end{proof}
   Now, we present the proof of \cref{result}, 
   \begin{proof}[Proof of \cref{result}]
       For $p^\text{ext} \leq \alpha_1 < \theta, D > 0$, from \cref{thm:exist}, the (SI) steady-state solution $p$ exists for any $L > 0$ and is unique, $p(x) \leq p^\text{ext} \leq \alpha_1$ for all $x \in (-L,L)$. Moreover, from \cref{convexity}, the reaction term $f$ has $f'(q) < 0$, for any $q \in (0,\alpha_1)$. Then, for any $x \in (-L,L)$, $f'(p(x)) \leq  0 < \lambda_1$. Hence, $p$ is asymptotically stable. 
       
       Besides, from Theorems \ref{thm:exist} and \ref{thm:nonSM}, for any $L > 0$ such that $L < M_d(p^\text{ext}, D) < M_*(p^\text{ext}, D)$, \cref{eqn:pb1} has neither (SD) nor non-(SM) steady-state solutions. So the (SI) steady-state solution is the unique steady-state solution. 
       
       Using a similar argument for the case $p^\text{ext} \geq \alpha_2$, we obtain the result in \cref{result}. 
    \end{proof}
\section{Application to the control of dengue vectors by introduction of the bacterium {\it Wolbachia}}
   \label{sec:bio}
   \subsection{Model}
   In this section, we show an application of our model to the control of mosquitoes using {\it Wolbachia}. Mosquitoes of genus {\it Aedes} are the vector of many dangerous arboviruses, such as dengue, zika, chikungunya and others. There exists neither effective treatment nor vaccine for these vector-borne diseases, and in such conditions, the main method to control them is to control the vector population. A biological control method using a bacterium called {\it Wolbachia} (see \cite{HOF}) was discovered and developed with this purpose. Besides reducing the ability of mosquitoes to transmit viruses, {\it Wolbachia} also causes an important phenomenon called {\it cytoplasmic incompatibility} (CI) on mosquitoes. More precisely, if a wild female mosquito is fertilized by a male carrying {\it Wolbachia}, its eggs almost cannot hatch. For more details about CI, we refer to \cite{WER}. In the case of {\it Aedes} mosquitoes, {\it Wolbachia} reduces lifespan, changes fecundity, and blocks the development of the virus. However, it does not influence the way mosquitoes move.
   
   In \cite{STR16}, model \cref{eqn:redu1}, \cref{eqn:redu2} was considered with $n_1 = n_i$ the density of the mosquitoes which are infected by {\it Wolbachia} and $n_2 = n_u$ the density of wild uninfected mosquitoes. Consider the following positive parameters:
   
   $\bullet$ $d_u, \delta d_u$: death rate of, respectively  uninfected mosquitoes and infected mosquitoes, $\delta > 1$ since {\it Wolbachia} reduces the lifespan of the mosquitoes;
   
   $\bullet$ $b_u, (1-s_f)b_u$: birth rate of, respectively uninfected mosquitoes and infected ones. Here $s_f \in [0,1)$ characterizes the fecundity decrease;
   
   $\bullet$ $s_h \in (0,1] $: the fraction of uninfected females' eggs fertilized by infected males that do not hatch, due to the cytoplasmic incompatibility (CI);
   	
   $\bullet$ $K$: carrying capacity, $A$: diffusion coefficient.
   
   Parameters $\delta, s_f, s_h$ have been estimated in several cases and can be found in the literature  (see \cite{BAR} and references therein). We always assume that $s_f < s_h$ (in practice, $s_f$ is close to 0 while $s_h$ is close to 1). 
   
   Several models have been proposed using these parameters. In the present study, a system of Lotka-Volterra type is proposed, where the parameter $\epsilon > 0$ is used to characterize the high fertility as follows. 
   \begin{equation}
   \begin{cases}
   \partial_t n_i^\epsilon - A\partial_{xx}n_i^\epsilon = (1-s_f)\dfrac{b_u}{\epsilon}n_i^\epsilon \left(1 - \dfrac{n_i^\epsilon + n_u^\epsilon}{K}\right) - \delta d_u n_i^\epsilon, \\
   \partial_t n_u^\epsilon - A\partial_{xx} n_u^\epsilon = \dfrac{b_u}{\epsilon}n_u^\epsilon \left( 1- s_h \dfrac{n_i}{n_i^\epsilon + n_u^\epsilon} \right) \left(1 - \dfrac{n_i^\epsilon + n_u^\epsilon}{K}\right) - d_u n_u^\epsilon,
   \end{cases}
   \label{eqn:model}
   \end{equation}
   where the reaction term describes birth and death. The factor $\left( 1- s_h \dfrac{n_i^\epsilon}{n_i^\epsilon + n_u^\epsilon} \right)$ characterizes the cytoplasmic incompatibility. Indeed, when $s_h = 1$, no egg of uninfected females fertilized by infected males can hatch, that is, there is complete cytoplasmic incompatibility. The factor becomes $\dfrac{n_u^\epsilon}{n_i^\epsilon + n_u^\epsilon}$ which means the birth rate of uninfected mosquitoes depends on the proportion of uninfected parents because only an uninfected couple can lay uninfected eggs. Whereas, $s_h = 0$ means that all the eggs of uninfected females hatch. In this case, the factor $\left( 1- s_h \dfrac{n_i^\epsilon}{n_i^\epsilon + n_u^\epsilon} \right)$ becomes $1$, so the growth rate of uninfected population is not altered by the pressure of the infected one.   
   
   In paper \cite{STR16}, the same model was studied in the entire space $\mathbb{R}$. In that case, the system \cref{eqn:model} has exactly two stable equilibria, namely the {\it Wolbachia} invasion steady state and the {\it Wolbachia} extinction steady state. In this paper, the authors show that when $\epsilon \rightarrow 0$ and the reaction terms satisfies some appropriate conditions, the proportion $p^\epsilon = \dfrac{n_i^\epsilon}{n_i^\epsilon + n_u^\epsilon}$ converges to the solution $p^0$ of the scalar equation $\partial_t p^0 - A\partial_{xx} p^0 = f(p^0)$, with the reaction term
   \begin{equation}
   f(p) = \delta d_u s_h \dfrac{p(1-p)(p-\theta)}{s_h p^2 - (s_f + s_h)p + 1},
   \label{eqn:reac}
   \end{equation} 	
   with $\theta = \dfrac{s_f + \delta -1}{\delta s_h}$. We will always assume that $s_f + \delta (1 - s_h) < 1$, so $\theta \in (0,1)$, and $f$ is a bistable function on $(0,1)$. The two stable steady states $1$ and $0$ of \cref{eqn:pb1} correspond to the success or failure of the biological control using {\it Wolbachia}. 	
   
   \subsection{Mosquito population in presence of migration}
   In this study, the migration of mosquitoes is taken into account. Typically, the inflow of wild uninfected mosquitoes and the outflow of the infected ones may influence the efficiency of the method using {\it Wolbachia}. Here, to model this effect, system \cref{eqn:model} is considered in a bounded domain with appropriate boundary conditions to characterize the migration of mosquitoes. In one-dimensional space, we consider $\Omega = (-L,L)$ and Robin boundary conditions as in \cref{eqn:redu2}
   \begin{equation}
   \begin{cases}
   \frac{\partial n_i^\epsilon}{\partial \nu} = -D(n_i^\epsilon - n_i^{\text{ext},\epsilon}) \qquad \text{ at } x = \pm L,\\
   \frac{\partial n_u^\epsilon}{\partial \nu} = -D(n_u^\epsilon - n_u^{\text{ext},\epsilon}) \qquad \text{ at } x = \pm L,
   \label{eqn:bord}
   \end{cases}
   \end{equation}	
   where $n_i^{\text{ext},\epsilon}, n_u^{\text{ext},\epsilon}$ do not depend on $t$ and $x$ but depend on parameter $\epsilon > 0$. Denote $p^\epsilon = \dfrac{n_i^\epsilon}{n_i^\epsilon + n_u^\epsilon}, n^\epsilon = \dfrac{1}{\epsilon} \left( 1 - \dfrac{n_i^\epsilon + n_u^\epsilon}{K} \right)$. In \cref{sec:converge}, we prove that when $\epsilon \rightarrow 0$, up to extraction of sub-sequences, $n^\epsilon$ converges weakly to $n^0 = h(p^0) $ for some explicit function $h$, and $p^\epsilon $ converges strongly towards solution $p^0$ of \cref{eqn:pb1} where $p^\text{ext}$ is the limit of $\dfrac{n_i^{\text{ext},\epsilon}}{n_i^{\text{ext},\epsilon} + n_u^{\text{ext},\epsilon}}$ when $\epsilon \rightarrow 0$, and the reaction term $f$ as in \cref{eqn:reac}. Function $f$ satisfies \cref{reaction} and \cref{convexity}, so the results in \cref{thm:exist} and \ref{thm:stability} can be applied to this problem. By changing spatial scale, we can normalize the diffusion coefficient into $A = 1$. 
   
   In this application, the parameters $L, D, p^\text{ext}$ correspond to the size of $\Omega$, the migration rate of mosquitoes, the proportion of infected mosquitoes surrounding the boundary. The main results in the present paper give information about existence and stability of equilibria depending upon different conditions for these parameters. Especially, from \cref{result}, we obtain that when the size $L$ of the domain is small, there exists a unique equilibrium for this problem and its values depend on the proportion of mosquitoes carrying \textit{Wolbachia} outside the domain ($p^\text{ext}$). More precisely, when $p^\text{ext}$ is small (i.e., $p^\text{ext} \leq \alpha_1$), solution of \cref{eqn:pb1} converges to the steady-state solution close to $0$, which corresponds to the extinction of mosquitoes carrying {\it Wolbachia}. Therefore, in this situation, the replacement strategy fails because of the migration through the boundary. Otherwise, when the proportion outside the domain is high (i.e., $p^\text{ext} \geq \alpha_2$), then the long-time behavior of solutions of \cref{eqn:pb1} has values close to $1$, which means that the mosquitoes carrying \textit{Wolbachia} can invades the whole population.  
   \subsection{Numerical illustration}
   \label{sec:simulation}
   In this section, we present the numerical illustration for the above results. Parameters are fixed according to biologically relevant data (adapted from \cite{FOC}). Time unit is the day, and parameters per day are in \cref{tab:parameter}. 
   \begin{table}
	    \caption{Parameters for the numerical illustration}
	    \label{tab:parameter}
	    \begin{center}
	      \begin{tabular}{| c | c c c c c c |}
            \hline
            Parameters & $b_u$ & $d_u$ & $\delta$ & $\sigma$ & $s_f$ & $s_h$ \\ [0.5ex] 
            \hline
            Values & 1.12 & 0.27 & $\frac{10}{9}$ & 1 & 0.1 & 0.8 \\ 
            \hline 
        \end{tabular}
	    \end{center}
	\end{table}
    Then, the reaction term $f$ in \cref{eqn:reac} has $\theta = 0.2375$, $\beta \approx 0.3633, \alpha_1 \approx 0.12, \alpha_2 \approx 0.7$. As proposed in section 3 of the modeling article \cite{OTR}, we may pick the value $830 m^2$ per day for the diffusivity of {\it Aedes} mosquitoes. Choose $A = 1$, so the $x$-axis unit in the simulation corresponds to $\sqrt{830/1} \approx 29$ m. 
   
   In the following parts, we check the convergence of $p^\epsilon$ when $\epsilon \rightarrow 0$ in \cref{convergence_num}. In \cref{steady_state_num}, corresponding to different parameters, we compute numerically the solutions of \cref{eqn:pb1} and \cref{eqn:pb2} to check their existence and stability. 
   
   \subsubsection{Convergence to the scalar equation}
   \label{convergence_num}
   Consider a mosquito population with large fecundity rate, that is, $ \epsilon \ll 1$. Model \cref{eqn:model} with boundary condition in \cref{eqn:bord} takes into account the migration of mosquitoes. 
   \begin{figure}
   	\centering
   	\begin{subfigure}{0.32\textwidth}
   		\centering
   		\includegraphics[width=\textwidth]{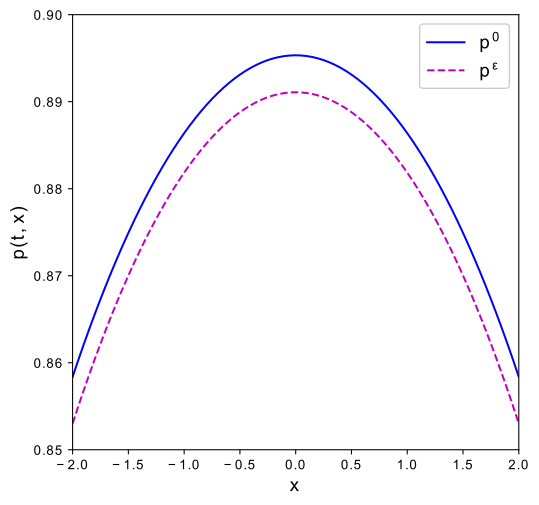}
   		\setlength{\abovecaptionskip}{0pt}
		\setlength{\belowcaptionskip}{0pt}
   		\caption{$\epsilon = 0.1$}
   	\end{subfigure}
   	\begin{subfigure}{0.32\textwidth}
   		\centering
   		\includegraphics[width=\textwidth]{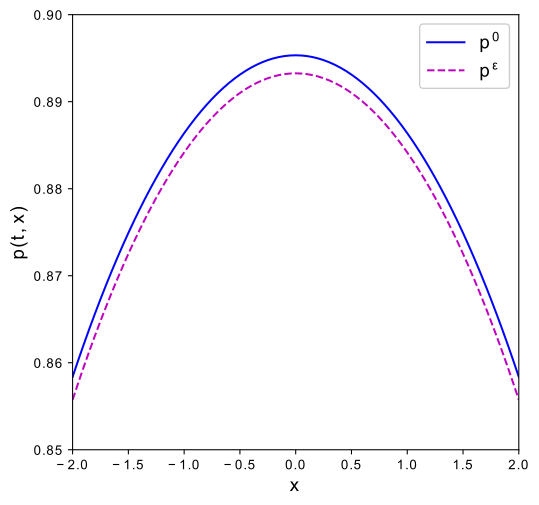}
   		\setlength{\abovecaptionskip}{0pt}
		\setlength{\belowcaptionskip}{0pt}
   		\caption{$\epsilon = 0.05$}
   	\end{subfigure}
   	\begin{subfigure}{0.32\textwidth}
   		\centering
   		\includegraphics[width=\textwidth]{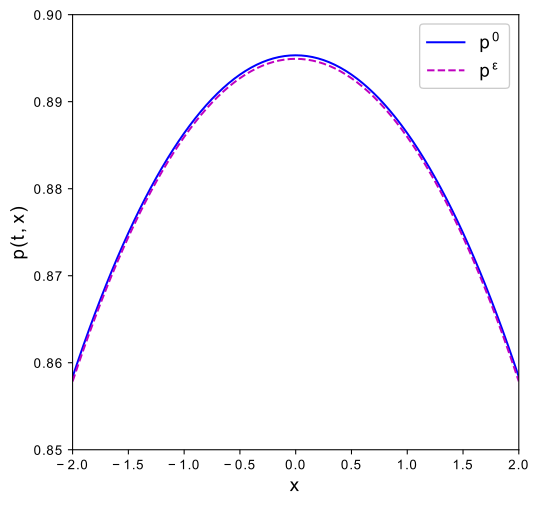}
   		\setlength{\abovecaptionskip}{0pt}
		\setlength{\belowcaptionskip}{0pt}
   		\caption{$\epsilon = 0.01$}
   	\end{subfigure}
   	\setlength{\abovecaptionskip}{5pt}
	\setlength{\belowcaptionskip}{0pt}
   	\caption{$p^0$ and $p^\epsilon$ at time $t = 50$ (days)}
   	\label{fig:epsilon}
   \end{figure}
   
   Fix $D = 0.05, p^\text{ext} = 0.1$ and $ L = 2$, the system \cref{eqn:model}, \cref{eqn:bord} is solved numerically thanks to a semi-implicit finite difference scheme with 3 different values of the parameters $\epsilon$. The initial data are chosen such that $n_i^\epsilon(t = 0) = n_u^\epsilon(t = 0)$, that is, $p^\text{init} = 0.5$. In \cref{fig:epsilon}, at time $t = 50$ days, the numerical solutions of \cref{eqn:pb1} are plotted with blue solid lines, the proportions $p^\epsilon = \dfrac{n_i^\epsilon}{n_i^\epsilon+n_u^\epsilon}$ are plotted with dashed lines. We observe that when $\epsilon$ goes to 0, the proportion $p^\epsilon$ converges to the solution $p^0$ of system \cref{eqn:pb1}. 
   \subsubsection{Steady-state solutions}
   \label{steady_state_num}
   For the different values of $p^\text{ext}$, the values of the integrals $\mathcal{F}_{1}$ and $\mathcal{F}_2$ as functions of $p(L)$ in \cref{eqn:F1} and \cref{eqn:F2} are plotted in \cref{fig:plotF}. For fixed values of $D$ and $p^\text{ext}$, \cref{fig:plotF} can play the role of bifurcation diagrams that show the relation between the value $p(L)$ of symmetric solutions $p$ and parameter $L$. Then, we can obtain the critical values of parameter $L$. Next, we compute numerically the (SM) steady-state solutions of \cref{eqn:pb1} with different values of $L > 0, D > 0, p^\text{ext} \in (0,1)$. 
   \begin{figure}
   	\centering
   	\begin{subfigure}{0.32\textwidth}
   		\centering
   		\includegraphics[width = \textwidth]{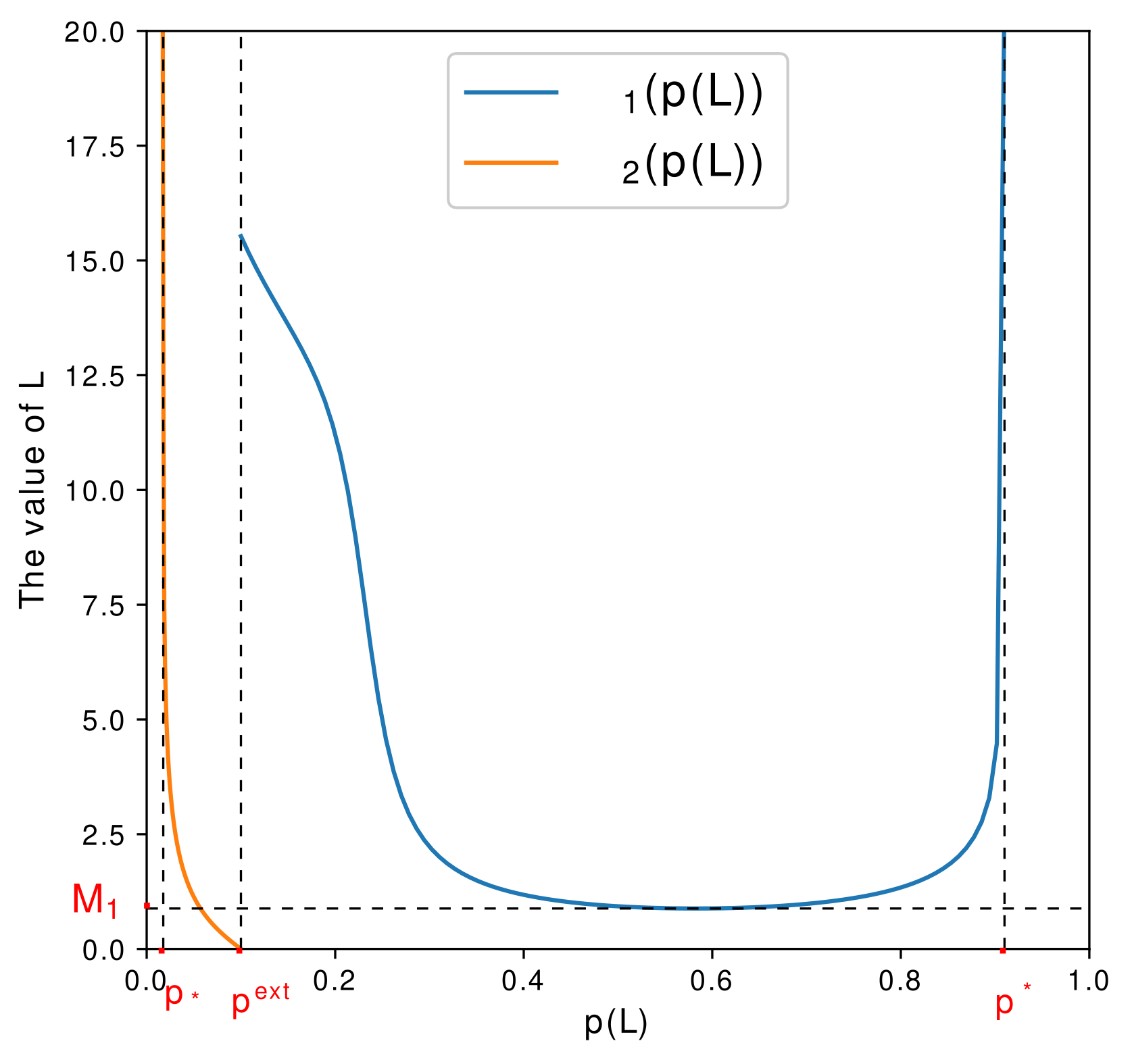}
   		\setlength{\abovecaptionskip}{0pt}
		\setlength{\belowcaptionskip}{0pt}
   		\caption{$p^\text{ext} = 0.1, D = 0.05$}
   		\label{fig:pext01}
   	\end{subfigure}
   	\begin{subfigure}{0.32\textwidth}
   		\centering
   		\includegraphics[width = \textwidth]{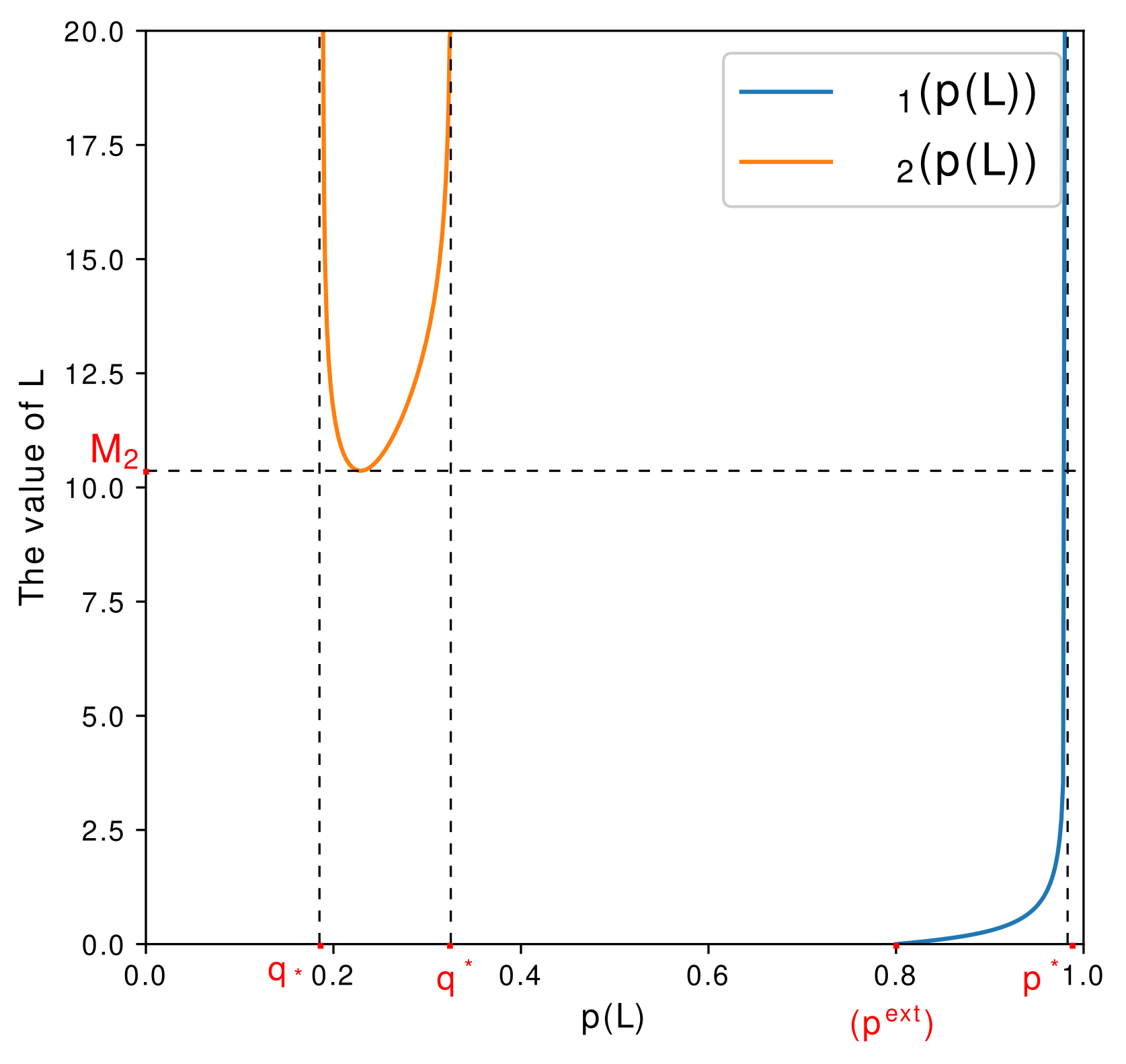}
   		\setlength{\abovecaptionskip}{0pt}
		\setlength{\belowcaptionskip}{0pt}
   		\caption{$p^\text{ext} = 0.8, D = 0.05$}
   		\label{fig:pext08}
   	\end{subfigure} 
   	\begin{subfigure}{0.32\textwidth}
   		\centering
   		\includegraphics[width = \textwidth]{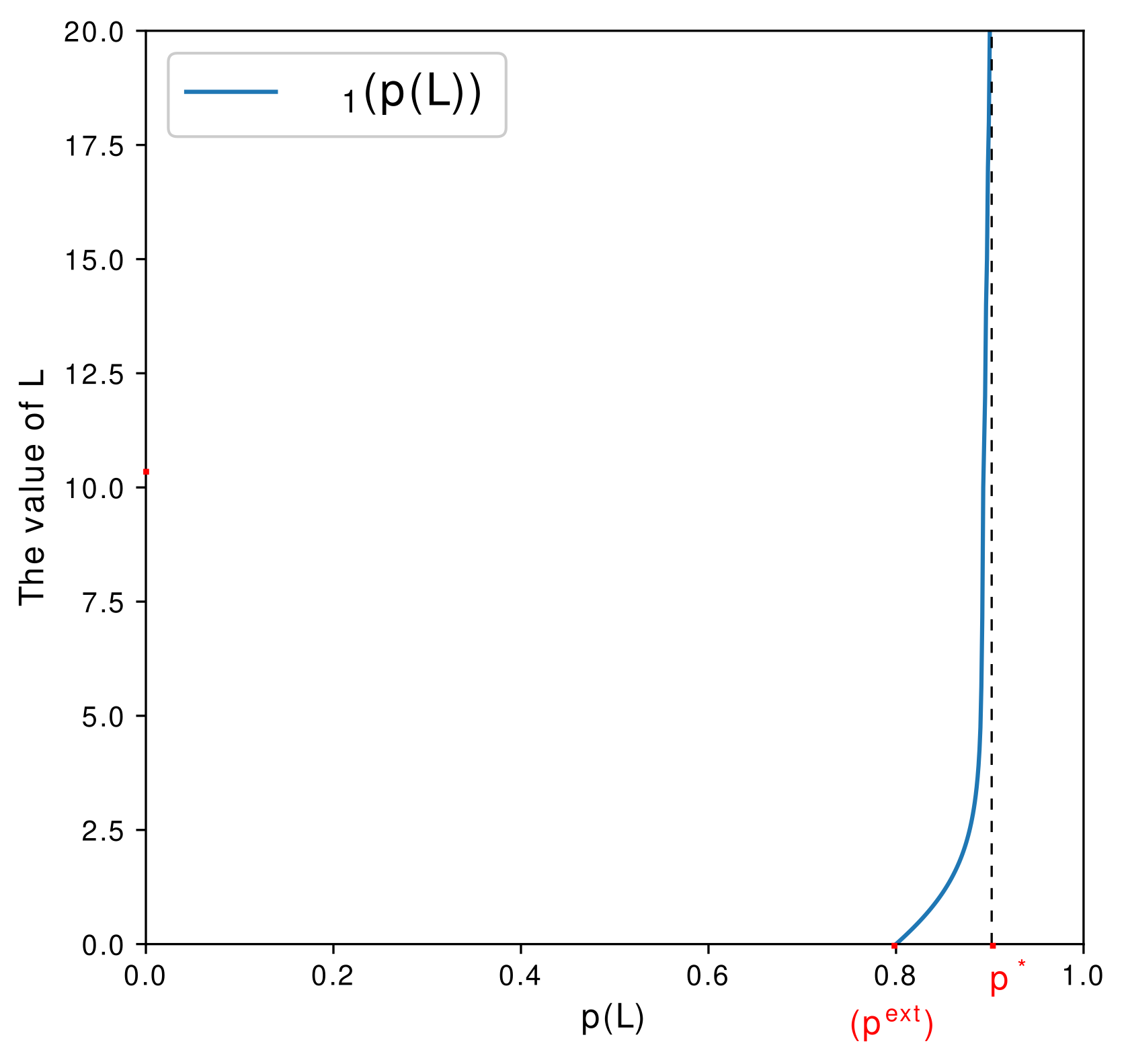}
   		\setlength{\abovecaptionskip}{0pt}
		\setlength{\belowcaptionskip}{0pt}
   		\caption{$p^\text{ext} = 0.8, D = 0.5$}
   		\label{fig:pext082}
   	\end{subfigure}
   	\setlength{\abovecaptionskip}{5pt}
	\setlength{\belowcaptionskip}{0pt}
   	\caption{Graphs of $\mathcal{F}_1$ and $\mathcal{F}_2$ with respect to $p(L)$.}
   	\label{fig:plotF}
   \end{figure}
   
   \noindent \textbf{Numerical method: } We use Newton method to solve equations $L = \mathcal{F}_{1,2}(p(L))$ and obtain the values of $p(L)$, then we can deduce the value of $p(0)$ by \cref{eqn:eq1}. Again by Newton method, we obtain $p(x)$ for any $x$ by solving $x = \displaystyle \int_{p(0)}^{p(x)} \dfrac{(-1)^k ds}{\sqrt{2F(p(0)) - 2 F(s)}}$. We also construct numerically a non-(SM) steady-state solution by the same technique but it is more sophisticated and details of the construction are omitted in this article for the sake of readability. 
   
   We also plot the time dynamics of solution $p^0(t,x)$ of \cref{eqn:pb1} at $t =10, 20, 40,60, 100$ to verify the asymptotic stability of steady-state solutions. Next, we consider different values of $p^\text{ext}$ and present our observation in each case.
   \begin{figure}
      \centering
   	\begin{subfigure}{0.32\textwidth}
   		\centering
   		\includegraphics[width = \textwidth]{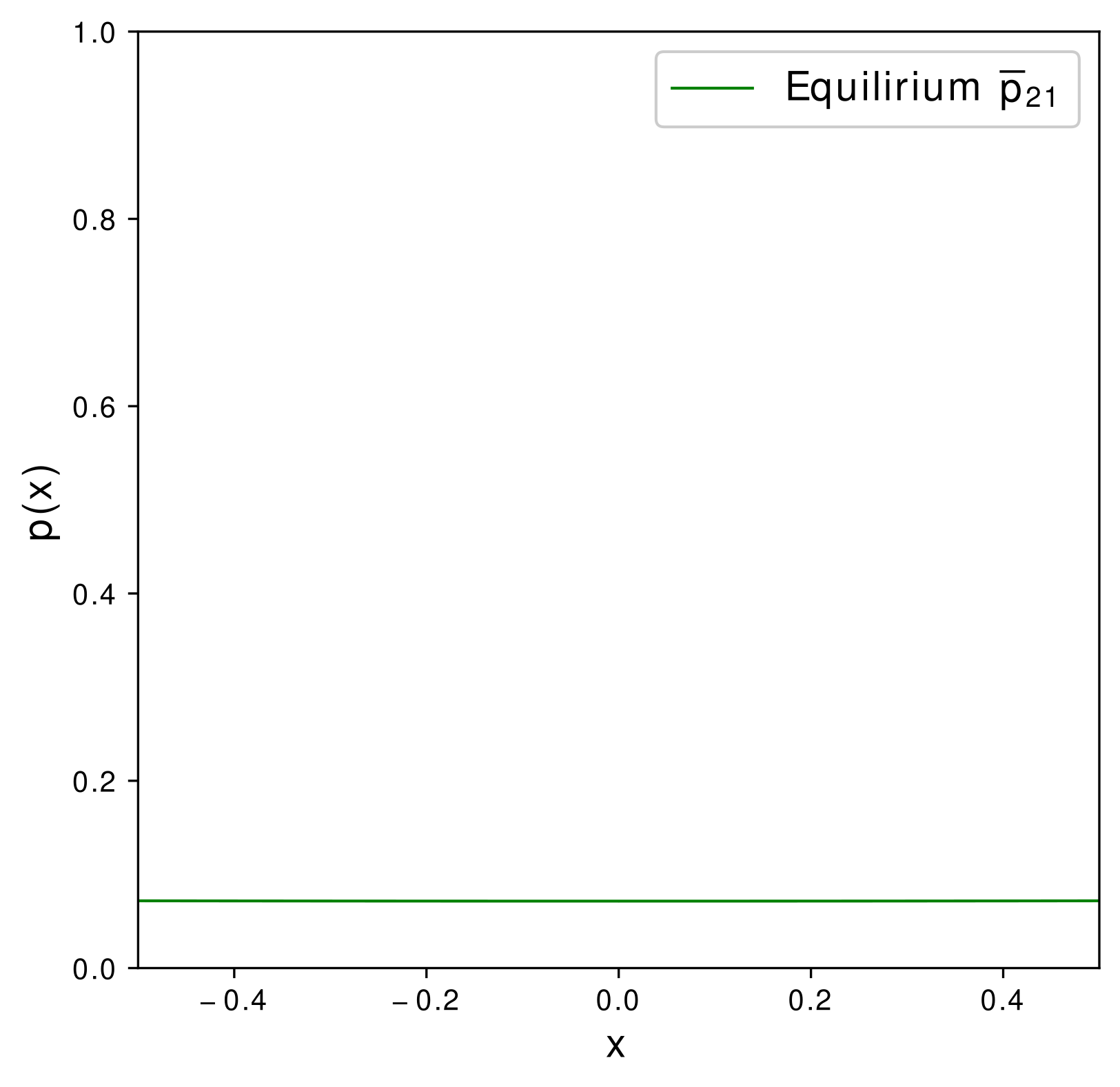} \\
   		\includegraphics[width = \textwidth]{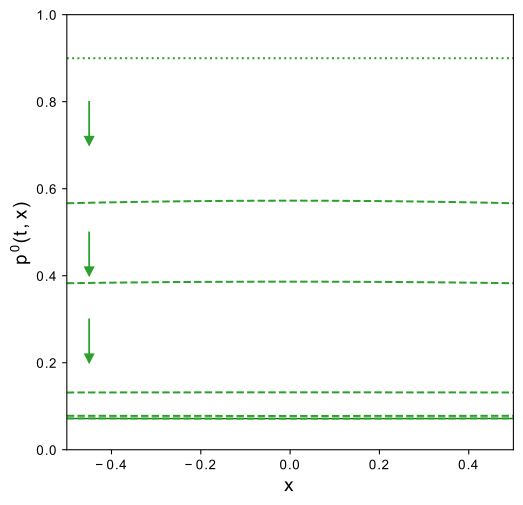} 
   		\caption{$L = 0.5 < M_1$}
   		\label{fig:04}
   	\end{subfigure}
   	\begin{subfigure}{0.32\textwidth}
   		\centering
   		\includegraphics[width = \textwidth]{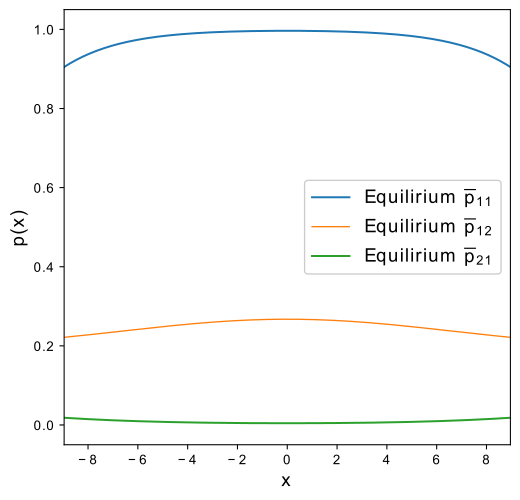} \\
   		\includegraphics[width = \textwidth]{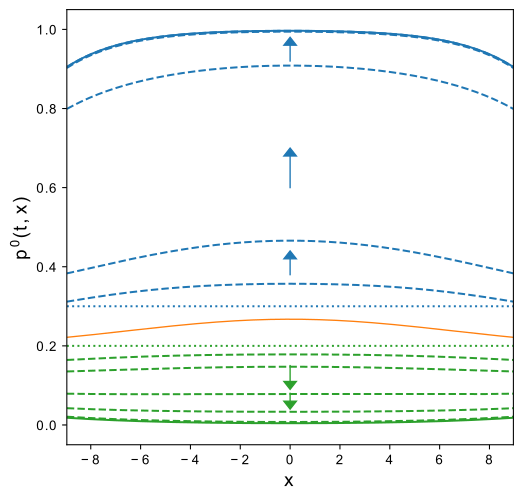}
   		\caption{$L = 8.96 > M_* > M_1$}
   		\label{fig:03}
   	\end{subfigure}
   	\begin{subfigure}{0.32\textwidth}
   		\centering
   		\includegraphics[width = \textwidth]{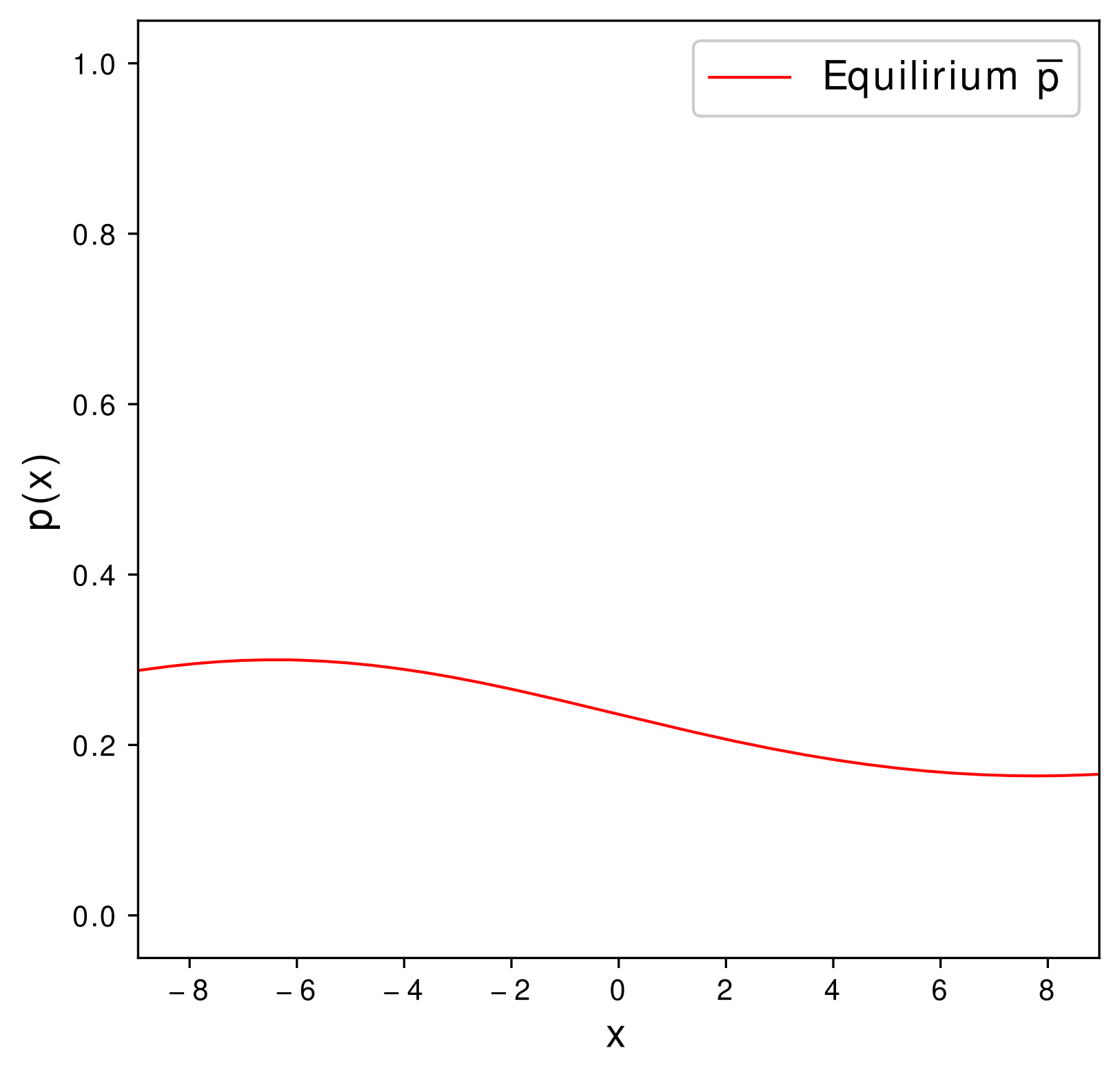} \\
   		\includegraphics[width = \textwidth]{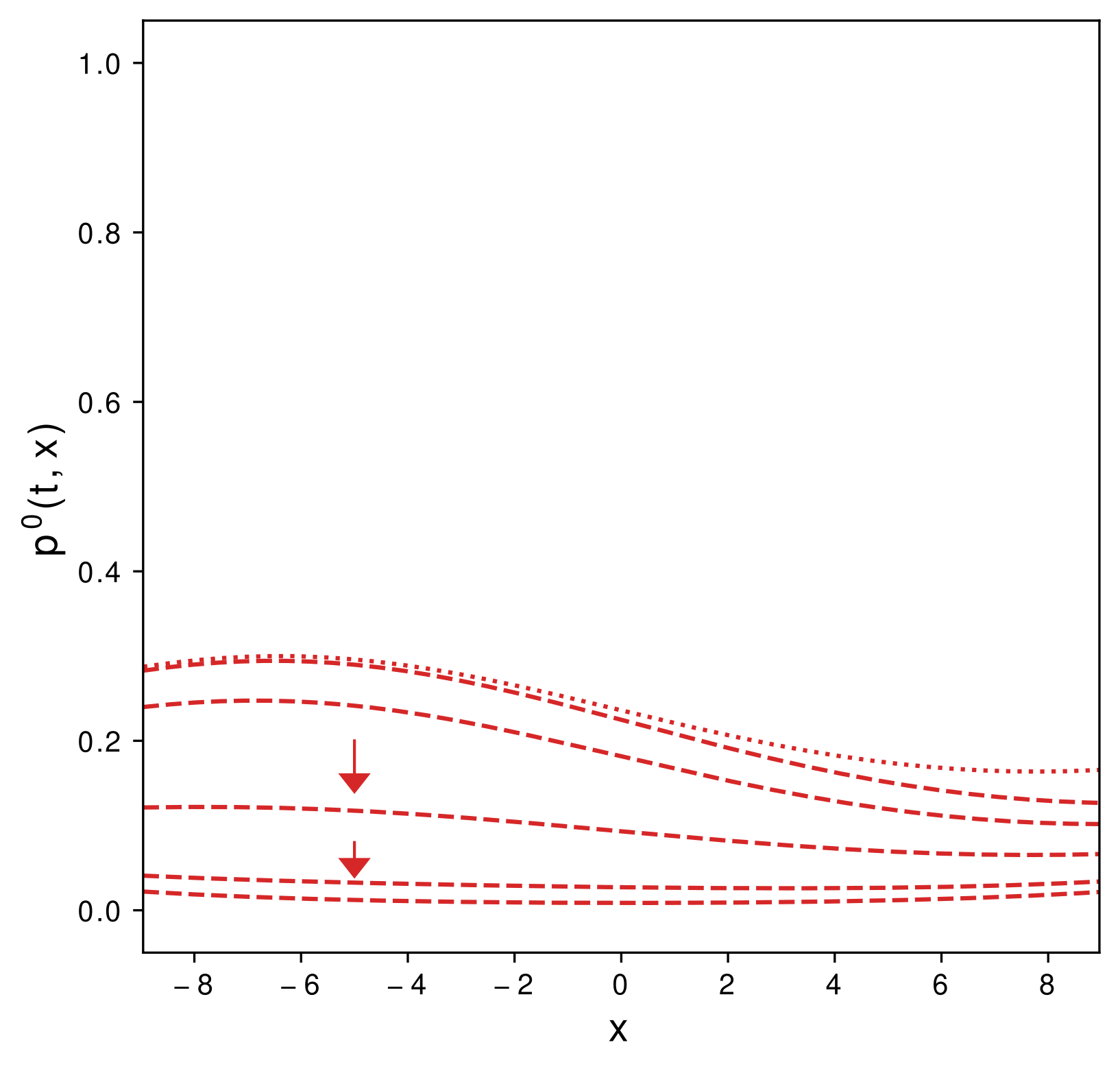}
   		\caption{$L = 8.96 > M_* > M_1$}
   		\label{fig:nonsym}
   	\end{subfigure}
   	\setlength{\abovecaptionskip}{0pt}
	\setlength{\belowcaptionskip}{-5pt}
   	\caption{Steady-state and time-dependent solutions when $p^\text{ext} = 0.1, D = 0.05$}
   	\label{fig:equi1}
   \end{figure}
   
   \noindent $\bullet$ { \bf Case 1: $p^\text{ext} = 0.1 < \alpha_1 $}.
   
   For $D = 0.05$ fixed, we observe in \cref{fig:pext01} that for any $L > 0$, equation $\mathcal{F}_2(p(L)) = L$ always admits exactly one solution. Thus, there always exists one (SI) steady-state solution with small values. We approximate that 
   \begin{displaymath}
        M_d(0.1,0.05) = M_1 \approx 0.8819, \quad M_*(0.1, 0.05) \approx 8.625.
   \end{displaymath}
   Also from \cref{fig:pext01}, we observe that when $L = M_1$, a bifurcation occurs and \cref{eqn:pb1} admits a (SD) steady-state solution, and when $L > M_1$ one can obtain two (SD) solutions. Moreover, when $L \geq M_*$, there exist non-symmetric steady-state solutions. We do numerical simulations for two values of $L$ as follows. 
   
   For $L = 0.5 < M_1$, the unique equilibrium $\overline{p}_{21}$ is (SI) and has values close to $0$ (see \cref{fig:04}). Solution $p^0$ of \cref{eqn:pb1} with any initial data converges to $\overline{p}_{21}$. This simulation is coherent with the asymptotic stability that we proved in \cref{result}.
   
   For $L = 8.96 > M_* > M_1$, together with $\overline{p}_{21}$, there exist two more (SD) steady-state solutions, namely $\overline{p}_{11}$, $\overline{p}_{12}$, (see \cref{fig:03}). This plot show that these steady-state solutions are ordered, and the time-dependent solutions converges to either the largest one $\overline{p}_{11}$ or the smallest one $\overline{p}_{21}$, while $\overline{p}_{12}$ with intermediate values is not an attractor. In \cref{fig:nonsym}, we find numerically a non-symmetric solution $\overline{p}$ of \cref{eqn:pb2} corresponding to orbit $T_3$ as in \cref{fig:phase}. Let the initial value $p^\text{init} \equiv \overline{p}$, then we observe from \cref{fig:nonsym} that $p^0$ still converges to the symmetric equilibrium $\overline{p}_{21}$.
   
    \noindent Moreover, the value $\lambda_1$ of \cref{thm:stability} in this case is approximately equal to $0.0063$. We also obtain that for any $x \in (-L,L)$,
    \begin{displaymath}
         f'(\overline{p}_{11}(x)) < 0,\quad f'(\overline{p}_{21}(x)) < 0,\quad f'(\overline{p}_{12}(x)) > 0.0462,\quad f'(\overline{p}(x)) > 0.022.
    \end{displaymath}
    Therefore, by applying \cref{thm:stability}, we deduce that the steady-state solutions $\overline{p}_{11}, \overline{p}_{21}$ are asymptotically stable, $\overline{p}_{12}$ and the non-symmetric equilibrium $\overline{p}$ are unstable. Thus, the numerical simulations in \cref{fig:equi1} are coherent to the theoretical results that we proved. 
   
   \noindent $\bullet$ { \bf Case 2: $p^\text{ext} = 0.8 > \alpha_2 > \beta$}. 
   
   In this case, we obtain $D_* \approx 0.16$. We present numerical illustrations for two cases: $D = 0.05 < D_*$ and $D = 0.5 > D_*$. 
   
   $\circ$ For $D = 0.05 < D_*$, we have $M_i(0.8, 0.05) = M_2 \approx 10.3646$ (see \cref{fig:pext08}). 
   
    For $L = 2 < M_2$, the unique equilibrium $\overline{p}_{11}$ is (SD) and has values close to $1$ (see \cref{fig:01}). The time-dependent solution $p^0$ of \cref{eqn:pb1} with any initial data converges to $\overline{p}_{11}$. This simulation is coherent to the asymptotic stability we obtained in \cref{result}. 
   
    For $L = 12 > M_2$, together with $\overline{p}_{11}$, there exist two more (SI) steady-state solutions, namely $\overline{p}_{21}$, $\overline{p}_{22}$, and they are ordered (see \cref{fig:02}). In this case, we obtain approximately that $\lambda_1 \approx 0.0063$ and for any $x \in (-L,L)$, one has 
   \begin{displaymath}
        f'(\overline{p}_{11}(x)) < 0, \quad f'(\overline{p}_{21}(x)) \in (-0.0398, 0.0368),  \quad f'(\overline{p}_{22}(x)) \in (-0.0195, 0.0673).
   \end{displaymath}
   By sufficient conditions in \cref{thm:stability}, we obtain that $\overline{p}_{11}$ is asymptotically stable but we can not conclude the stability for $\overline{p}_{21}$ and $\overline{p}_{22}$. The time dynamics of $p^0$ in \cref{fig:02} suggests that the smallest steady-state solution $\overline{p}_{21}$ is asymptotically stable and $\overline{p}_{22}$ seems to be unstable.  
   
   $\circ$ For $D = 0.5 > D_*$, function $\mathcal{F}_2$ is not defined (see \cref{fig:pext082}), so problem \cref{eqn:pb2} admits only one (SD) steady-solution, and we obtain that it is unique and asymptotically stable (see \cref{fig:00}). 
   \begin{figure}
      \centering
   	\begin{subfigure}{0.32\textwidth}
   		\centering
   		\includegraphics[width = \textwidth]{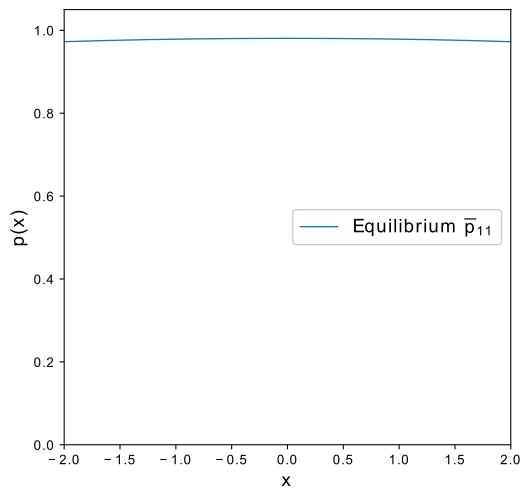} \\
   		\includegraphics[width = \textwidth]{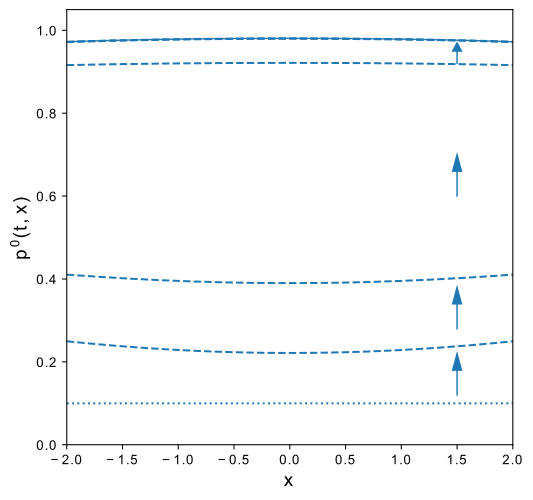} 
   		\caption{$L = 2, D = 0.05 < D_*$}
   		\label{fig:01}
   	\end{subfigure}
   	\begin{subfigure}{0.32\textwidth}
   		\centering
   		\includegraphics[width = \textwidth]{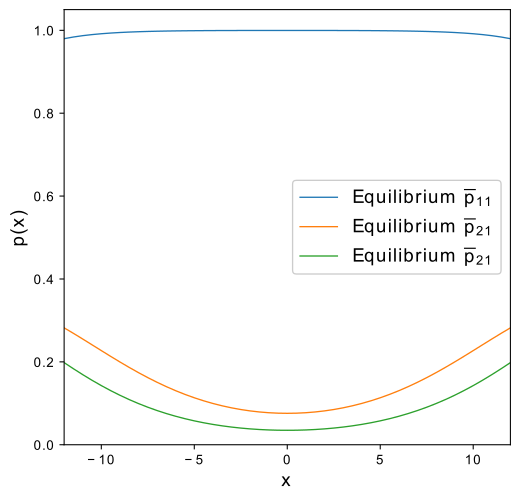} \\
   		\includegraphics[width = \textwidth]{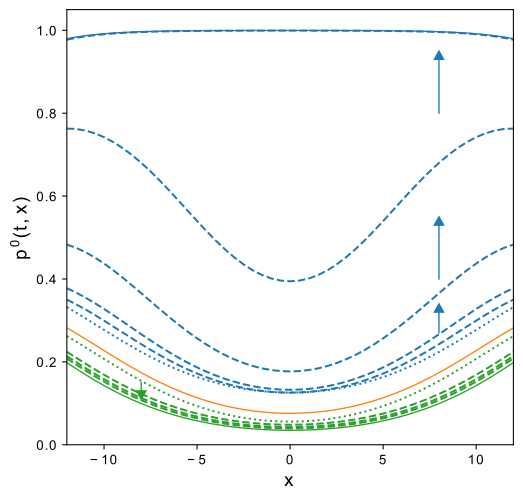}
   		\caption{$L = 12, D = 0.05 < D_*$}
   		\label{fig:02}
   	\end{subfigure}
   	\begin{subfigure}{0.32\textwidth}
   		\centering
   		\includegraphics[width = \textwidth]{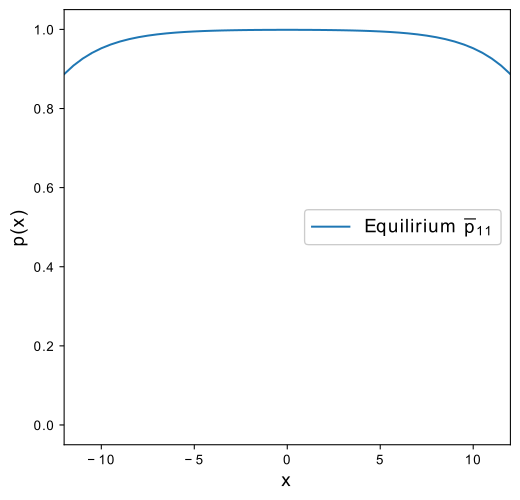} \\
   		\includegraphics[width = \textwidth]{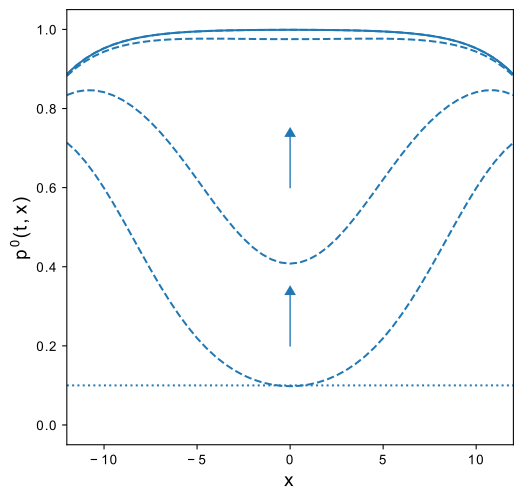}
   		\caption{$L = 12, D = 0.5 > D_*$}
   		\label{fig:00}
   	\end{subfigure}
   	\setlength{\abovecaptionskip}{5pt}
	\setlength{\belowcaptionskip}{0pt}
   	\caption{Steady-state and time-dependent solutions when $p^\text{ext} = 0.8$}
   	\label{fig:equi2}
   \end{figure}
 
	\section{Conclusion and perspectives}
	We have studied the existence and stability of steady-state solutions with values in $[0,1]$ of a reaction-diffusion equation 
	\begin{displaymath}
	     \partial_t p - \partial_{xx}p = f(p)
	\end{displaymath}
	on an interval $(-L,L)$ with cubic nonlinearity $f$ and inhomogeneous Robin boundary conditions
	\begin{displaymath}
	     \dfrac{\partial p}{\partial \nu} = D(p-p^\text{ext}),
	\end{displaymath}
	where constant $p^\text{ext} \in (0,1)$ is an analogue of $p$, and constant $D > 0$. We have shown how the analysis of this problem depends on the parameters $p^\text{ext}$, $D$, and $L$. More precisely, the main results say that there always exists a symmetric steady-state solution that is monotone on each half of the domain. For $p^\text{ext}$ large, the value of this steady-state solution is close to 1, otherwise, it is close to $0$. Besides, the larger value of $L$, the more steady-state solutions this problem admits. We have found the critical values of $L$ so that when the parameters surpass these critical values, the number of steady-state solutions increases. We also provided some sufficient conditions for the stability and instability of the steady-state solutions. 
	
	We presented an application of our results on the control of dengue vector using {\it Wolbachia} bacterium that can be transmitted maternally. Since {\it Wolbachia} can help reduce vectorial capacity of the mosquitoes, the main goal of this method is to replace wild mosquitoes by mosquitoes carrying {\it Wolbachia}. In this application, we considered $p$ as the proportion of mosquitoes carrying {\it Wolbachia} and used the equation above to model the dynamic of the mosquito population. The boundary condition describes the migration through the border of the domain. This replacement method only works when $p$ can reach an equilibrium close to $1$. Therefore, the study of existence and stability of the steady-state solution close to $1$ is meaningful and depends strongly on the parameters $p^\text{ext}$, $D$, and $L$. In realistic situations, the proportion $p^\text{ext}$ of mosquitoes carrying {\it Wolbachia} outside the domain is usually low. Using the theoretical results proved in this article, one sees that, to have major chances of success, one should try to treat large regions ($L$ large), well isolated ($D$ small) and possibly applying a population replacement method in a zone outside $\Omega$ (to increase $p^\text{ext}$ by reducing its denominator). 
	
	As a natural continuation of the present work, higher dimension problems and more general boundary conditions can be studied. In more realistic cases, $p^\text{ext}$ can be considered to depend on space and the periodic solutions can be the next problem for our study. Besides, when an equilibrium close to $1$ exists and is stable, one may consider multiple strategies using multiple releases of mosquitoes carrying {\it Wolbachia}. To optimize the number of mosquitoes released to guarantee the success of this method under the difficulties enlightened by this paper is an interesting problem for future works.  
	
	\appendix
	\section{Asymptotic limit of reaction-diffusion systems}
	    \label{sec:converge}
		In \cite{STR16}, the authors reduced a 2-by-2 reaction-diffusion system of Lotka-Volterra type modeling two biological populations to a scalar equation as in \cref{eqn:pb1} when the fecundity rate is very large. This limit problem was first proved in the whole domain. In the present study, we prove the limit for a system in a bounded domain with inhomogeneous Robin boundary conditions. In the following part, we recall the necessary assumptions and present results about this problem. 
		
		Although the main result of the paper is in one-dimensional space, the following result holds in any dimension $d$. Let $\Omega \subset \mathbb{R}^d$ be a bounded domain and consider the initial-boundary-value problem \cref{eqn:pb3} depending on parameter $\epsilon > 0$, 
			\begin{equation}
				\begin{cases}
					\partial_t n_1^\epsilon - \Delta n_1^\epsilon = n_1^\epsilon f_1^\epsilon(n_1^\epsilon,n_2^\epsilon) & \text{ in } (0,T) \times \Omega, \\
					\partial_t n_2^\epsilon - \Delta n_2^\epsilon = n_2^\epsilon f_2^\epsilon(n_1^\epsilon,n_2^\epsilon) & \text{ in } (0,T) \times \Omega, \\
					n_1^\epsilon(0,\cdot) = n_1^{\text{init},\epsilon}, \quad n_2^\epsilon(0,\cdot) = n_2^{\text{init},\epsilon} & \text{ in } \Omega,\\
					\frac{\partial n_1^\epsilon}{\partial \nu} = -D(n_1^\epsilon - n_1^{\text{ext},\epsilon}), \quad \frac{\partial n_2^\epsilon}{\partial \nu} = -D(n_2^\epsilon - n_2^{\text{ext},\epsilon}) & \text{ on } (0,T) \times \partial \Omega ,\\
				\end{cases}
			\label{eqn:pb3}
			\end{equation}
		where we assume that $f_1^\epsilon, f_2^\epsilon$ are smooth enough to guarantee existence and uniqueness of a classical solution for fixed $\epsilon$. More precisely, the following assumptions are made:
		\begin{assumption}{\normalfont (Initial and boundary conditions).}
			\label{a1}
			$n_1^{\text{init},\epsilon}, n_2^{\text{init},\epsilon} \in L^\infty(\Omega)$ with $ n_1^{\text{init},\epsilon}, n_2^{\text{init},\epsilon} \geq 0$ and $n_2^{\text{init},\epsilon}$ is not identical to $0$. 
			
			$D > 0$ is constant, $n_1^{\text{ext},\epsilon} \geq 0,n_2^{\text{ext},\epsilon} > 0$ do not depend on time $t$ and position $x$.
		\end{assumption}
	To study the limit problem, we define the "rescaled total population" $n^\epsilon$ and proportion $p^\epsilon$, by
	\begin{equation}
		n^\epsilon := \dfrac{1}{\epsilon} - n_1^\epsilon - n_2^\epsilon, \quad p^\epsilon = \dfrac{n_1^\epsilon}{n_1^\epsilon + n_2^\epsilon}.
	\end{equation}
	Next, we recall some assumptions that were proposed in \cite{STR16} on the families of functions $(f_1^\epsilon, f_2^\epsilon)_{\epsilon > 0}$ to study the convergence of $p^\epsilon$ when $\epsilon \rightarrow 0$
	\begin{assumption}
		\label{a2}
		Function $f_1^\epsilon, f_2^\epsilon$ are of class $\mathcal{C}^2(\mathbb{R}^2_+ \{ 0\})$, and for $i \in \{1,2\}$ there exists $F_i \in \mathcal{C}^2(\mathbb{R}^2)$ (independent of $\epsilon$) such that 
		\begin{equation}
			f_i^\epsilon(n_1^\epsilon,n_2^\epsilon) = F_i(n^\epsilon,p^\epsilon).
		\end{equation}
	\end{assumption}
	That is, we may write $f_i^\epsilon(n_1^\epsilon,n_2^\epsilon) = F_i\left(\frac{1}{\epsilon} - n_1^\epsilon - n_2^\epsilon, \frac{n_1^\epsilon}{n_1^\epsilon + n_2^\epsilon}\right)$ for $i \in \{1,2\}$. 
	
	Then, we can deduce that $p^\epsilon$ and $n^\epsilon$ satisfy the following system 
	\begin{equation}
		\begin{cases}
			\partial_t n^\epsilon - \Delta n^\epsilon = -(\frac{1}{\epsilon} - n^\epsilon) \left[p^\epsilon F_1(n^\epsilon,p^\epsilon) + (1-p^\epsilon)F_2(n^\epsilon,p^\epsilon)\right] & \text{ in } (0,T) \times \Omega, \\
			\partial_t p^\epsilon - \Delta p^\epsilon + \frac{2\epsilon A}{1- \epsilon n^\epsilon} \nabla p^\epsilon \cdot \nabla n^\epsilon = p^\epsilon(1-p^\epsilon)(F_1-F_2)(n^\epsilon,p^\epsilon) & \text{ in } (0,T) \times \Omega, \\
			n^\epsilon(0,\cdot) = n^{\text{init},\epsilon}, \quad p^\epsilon(0,\cdot) = p^{\text{init},\epsilon} & \text{ in } \Omega,\\
			\frac{\partial n^\epsilon}{\partial \nu} = -D(n^\epsilon - n^{\text{ext},\epsilon}) & \text{ on } (0,T) \times \partial \Omega ,\\
			\frac{\partial p^\epsilon}{\partial \nu} = -D(p^\epsilon - p^{\text{ext},\epsilon})\frac{1 - \epsilon n^{\text{ext},\epsilon}}{1 - \epsilon n^{\epsilon}}& \text{ on } (0,T) \times \partial \Omega ,\\
		\end{cases}
	\label{eqn:pb4}
	\end{equation}
	where $(F_1-F_2)(n^\epsilon,p^\epsilon) = F_1(n^\epsilon,p^\epsilon) - F_2(n^\epsilon,p^\epsilon)$, and 
	\begin{equation}
	    n^{\text{init},\epsilon} := \dfrac{1}{\epsilon} - n_1^{\text{init},\epsilon} - n_2^{\text{init},\epsilon}, \quad p^{\text{init},\epsilon} := \dfrac{n_1^{\text{init},\epsilon}}{n_1^{\text{init},\epsilon} + n_2^{\text{init},\epsilon}},
	\end{equation}
	\begin{equation}
	    n^{\text{ext},\epsilon} := \dfrac{1}{\epsilon} - n_1^{\text{ext},\epsilon} - n_2^{\text{ext},\epsilon}, \quad p^{\text{ext},\epsilon} := \dfrac{n_1^{\text{ext},\epsilon}}{n_1^{\text{ext},\epsilon} + n_2^{\text{ext},\epsilon}}.
	\end{equation}
	
	Let us denote $H(n,p) = -p F_1(n,p) - (1-p)F_2(n,p)$. The following assumption guarantees existence of zeros of $H$ given by $(n, p) = (h(p),p)$ for each $p\in [0,1]$. 
	\begin{assumption} 
		\label{a3}
		In addition to Assumption \ref{a2}, 
		
		(i) $\exists B > 0$ such that $\forall n \geq 0, \forall p \in [0,1]$, $\partial_n H(n,p) \leq -B$,
		
		(ii) $\forall p >0 $, $H(0,p) > 0$.
	\end{assumption}
	Conditions (i) and (ii) imply that for all $p \in [0,1]$, there exists a unique $n =: h(p) \in \mathbb{R}_+^*$ such that $H(n,p) = 0$.  We have $H \in \mathcal{C}^2(\mathbb{R}^2_+)$ (from \cref{a2}) thus $h \in \mathcal{C}^2(0,1)$, with $H(h(p),p)= 0$ for all $p \in [0,1]$.  
	
    The following assumptions are made for the initial data and boundary conditions
	\begin{assumption}
		\label{a4}
		There exists a function $p^\text{init} \in L^2(\Omega)$ such that $p^{\text{init}, \epsilon} \displaystyle \underset{\epsilon \rightarrow 0}{\rightharpoonup}   p^\text{init}$ weakly in $L^2(\Omega)$. 
		Function $n^{\text{init},\epsilon} - h(0) \in L^2 \cap L^\infty(\Omega)$ is uniformly bounded in $\epsilon > 0$. 
	\end{assumption}
	\begin{assumption}
		\label{a5}
		There exists positive constants $\tilde{\epsilon} > 0, \tilde{K} > 0$ such that for any $\epsilon \in (0, \tilde{\epsilon})$, we have $|n^{\text{ext},\epsilon}| < \widetilde{K}$. 
		
		There exists a constant $p^\text{ext} \in (0,1)$ not depending on $\epsilon$ such that $p^{\text{ext},\epsilon} \underset{\epsilon \rightarrow 0}{\rightarrow} p^\text{ext}$ 
	\end{assumption}
	
    \noindent {\bf Convergence result.} For fixed $\epsilon > 0$, existence of solutions of \cref{eqn:pb4} is classical (see, e.g. \cite{PER}). Following the idea in \cite{STR16}, we present the asymptotic limit of the proportion $p^\epsilon$ and $n^\epsilon$ in the following theorem.
	
	\begin{theorem}
		\label{convergence}
		Assume that Assumptions \ref{a1}-\ref{a5} are satisfied and consider the solution $(n^\epsilon, p^\epsilon)$ of \cref{eqn:pb4}. Then, for all $T > 0$, we have the convergence
		
		$ \begin{cases}
			p^\epsilon \xrightarrow[\epsilon \rightarrow 0]{} p^0 \text{ strongly in } L^2(0,T;L^2(\Omega)), \text{ weakly in } L^2(0,T;H^1(\Omega)), \\
			n^\epsilon - h(p^\epsilon) \xrightarrow[\epsilon \rightarrow 0]{} 0 \text{ strongly in } L^2(0,T;L^2(\Omega)), \text{ weakly in } L^2(0,T;H^1(\Omega)),
		\end{cases}
		$
		
	\noindent where $p^0$ is the unique solution of 
	\begin{equation}
		\begin{cases}
			\partial_t p^0 - \Delta p^0
			= p^0(1-p^0)(F_1-F_2)(h(p^0),p^0), & \text{ in } (0,T) \times \Omega, \\
			p^0(0,\cdot) = p^\text{init} & \text{ in } \Omega \\
		    \frac{\partial p^0}{\partial \nu} = -D(p^0 - p^\text{ext}) & \text{ on } (0,T)\times \partial  \Omega.
		\end{cases}
	\end{equation}
	\end{theorem}
	
	We recall the apriori estimates of \cite{STR16} without proof and present some bounds on the boundary in \cref{estimate}. Then we use the Aubin-Lions lemma and trace theorem to prove the limit in \cref{limit}.  
	\subsection{Uniform a priori estimates}
    \label{estimate}
    First, we establish the uniform bound with respect to $\epsilon$ in $L^\infty$ in the following lemma
    \begin{lemma}
    	\label{Linfty}
    	Under Assumptions \ref{a1}-\ref{a5}, for a given value $\epsilon > 0$, let $(n^\epsilon, p^\epsilon)$ be the unique solution of \cref{eqn:pb4}. Then, for any $T > 0$, $0 \leq p^\epsilon \leq 1$ in $[0,T] \times \overline{\Omega}$ for all $\epsilon > 0$. Also, there exists $\epsilon_0 > 0, K_0 > 0$ such that for any $\epsilon \in (0,\epsilon_0)$, $|| n^\epsilon||_{L^\infty([0,T] \times \Omega)} \leq K_0$. 
    	
    	Moreover, $n^\epsilon$ is uniformly bounded on $[0,T] \times \partial \Omega$. 
    \end{lemma}
    \begin{proof}
    Using the same method as in Lemma 5 of \cite{STR16}, we obtain the uniform bounds for $p^\epsilon$ in $[0,T] \times \overline{\Omega}$, and for $n^\epsilon$ in $L^\infty([0,T] \times \Omega)$. 
    
	Moreover, for any $ x \in \partial \Omega$, let $\nu$ be the normal outward vector through $x$. Then, for $\delta > 0$ small enough, $x - \delta \nu \in \Omega$. From the boundary condition for $n^\epsilon$ in \cref{eqn:pb4}, one has for $t \in [0,T]$, $\displaystyle \lim_{\delta \rightarrow 0^+} \dfrac{n^\epsilon(t,x) - n^\epsilon(t,x-\delta \nu)}{\delta} = -D(n^\epsilon(t,x) - n^{\text{ext},\epsilon})$.  
	
	So for any $\eta > 0$, there exists $\delta > 0$ small such that 
	\begin{center}
	    $\left|\dfrac{n^\epsilon(t,x) - n^\epsilon(t,x-\delta \nu)}{\delta} + D(n^\epsilon(t,x) - n^{\text{ext},\epsilon})\right| \leq \eta$.
	\end{center}
	
	Thus, $n^\epsilon(t,x) (1+\delta D) \leq n^\epsilon(t,x-\delta \nu) + \delta D n^{\text{ext},\epsilon} + \delta \eta$, then for $\eta$ and $\delta$ small enough, for any $\epsilon < \epsilon_0$, $t \in [0,T], x \in \partial \Omega $, since $x - \delta\nu \in \Omega$, one has $|n^\epsilon(t,x)| \leq K_0 + \delta D \widetilde{K} + \delta \eta < K_1$. Then $n^\epsilon$ is uniformly bounded on $[0,T] \times \partial \Omega$ and $||n^\epsilon||_{L^\infty([0,T] \times \partial \Omega)} \leq K_1$. 
    \end{proof}
    
    The following lemmas can be proved analogously to the proof in \cite{STR16}. 
    \begin{lemma}
    	\label{H1}
    	Under Assumptions \ref{a1}-\ref{a5}, for $\epsilon > 0$ small enough, let $(n^\epsilon, p^\epsilon)$ be the unique solution of \cref{eqn:pb4}. We have the following uniform estimates 
    	\begin{equation}
    		\displaystyle \epsilon \int_{0}^{T} \int_\Omega |\nabla n^\epsilon|^2 dxdt \leq C_0,\\
    		\displaystyle \int_{0}^{T} \int_\Omega |\nabla p^\epsilon|^2 dxdt \leq \overline{C},
    		\label{eqn:pepsilon}
    	\end{equation}
    for some positive constants $C_0$ and $\overline{C}$. 
    \end{lemma}
    
    Denote $M^\epsilon := n^\epsilon - h(p^\epsilon)$ where $h$ is defined in \cref{a3}. The following provide the convergence of $M^\epsilon$.
    \begin{lemma}
    \label{nconverge}
        Let $T > 0$, under Assumptions \ref{a1}-\ref{a5}, one has $M^\epsilon \displaystyle \rightarrow 0 $ in $L^2(0,T; L^2(\Omega)) $ when $\epsilon \rightarrow 0$. 
    \end{lemma}
    Now, we provide a uniform estimate for $\partial_t p^\epsilon$ with respect to $\epsilon$ in the following lemma. 
	\begin{lemma}
		\label{X'}
		Under Assumptions \ref{a1}-\ref{a5}, for $\epsilon > 0$ small enough, $\partial_t p^\epsilon$ is uniformly bounded in $L^2(0,T;X')$ with respect to $\epsilon$, where $ X  = H^1(\Omega) \cap L^\infty(\Omega)$. 
	\end{lemma}
	\subsection{Proof of convergence}
	\label{limit}
	The idea to prove \cref{convergence} is relied on the relative compactness obtained from the Aubin-Lions lemma below (see \cite{SIM}) 
	\begin{lemma}[Aubin-Lions]
		\label{AL}
		Let $T > 0$, $q \in (1,\infty)$, and $(\psi_n)_n$ a bounded sequence in $L^q(0,T;B)$, where $B$ is a Banach space. If $(\psi_n)$ is bounded in $L^q(0,T;X)$ and $X$ embeds compactly in $B$, and if $(\partial_t\psi_n)_n$ is bounded in $L^q(0,T;X')$ uniformly with respect to $n$, then $(\psi_n)_n$ is relatively compact in $L^q(0,T;B)$. 
	\end{lemma}
	\begin{proof}[\bf Proof of \cref{convergence}]
	    We use 3 steps to proof \cref{convergence}. First, we obtain the relative compactness of $(p^\epsilon)$ by applying Aubin-Lions lemma, and prove that there exists (up to extracting subsequences) a limit function. Then, we study its behavior on the boundary using the trace theorem. Finally, thanks to our uniform bounds, we show that the limit function satisfies a problem whose solution is unique. 
		
		{\bf Step 1: }In our problem, we need to apply the Lions-Aubin lemma with $q= 2, B = L^2(\Omega)$ and $X = H^1(\Omega) \cap L^\infty(\Omega)$ to $(\psi_\epsilon) = (p^\epsilon)_\epsilon$. The compact embedding from $X$ to $B$ is valid by the Rellich-Kondrachov theorem. In the previous section, we have already obtained uniform estimates that are sufficient to apply the Aubin-Lions lemma. The sequence $(p^\epsilon)_\epsilon $ is bounded in $L^2(0,T;L^2(\Omega))$ due to \cref{Linfty}
		\begin{center}
		    $|| p^\epsilon ||^2_{L^2(0,T;L^2(\Omega))} = \displaystyle \int_{0}^{T} \int_\Omega |p^\epsilon|^2 dx dt \leq || p^\epsilon ||^2_{L^\infty(0,T;L^2(\Omega))} \text{ meas}(\Omega) T < \infty,$
		\end{center}
		for $\epsilon < \epsilon_0$ small enough. Then, due to \cref{H1}, this sequence is bounded in $L^2(0,T;X)$.
		The sequence $(\partial_t p^\epsilon)_\epsilon$ is bounded in $L^2(0,T;X')$ by \cref{X'}. Thus, we can apply Aubin-Lions lemma and deduce that $(p^\epsilon)_\epsilon$ is strongly relatively compact in $L^2(0,T;L^2(\Omega))$. Therefore, there exists $p^0 \in L^2(0,T;H^1(\Omega))$ such that, up to extraction of subsequences, we have $p^\epsilon \rightarrow p^0$ strongly in $L^2((0,T) \times \Omega)$ and a.e., $\nabla p^\epsilon \rightharpoonup \nabla p^0$ weakly in $L^2((0,T) \times \Omega)$. 
		
		Moreover, by the triangle inequality we have $|n^\epsilon - h(p^0)| \leq |n^\epsilon - h(p^\epsilon)| + |h(p^\epsilon) - h(p^0)| \leq |n^\epsilon - h(p^\epsilon)| + ||h'||_{L^\infty([0,1])}|p^\epsilon - p^0|$. From the strong convergence of $p^\epsilon$ and $M^\epsilon$ in \cref{nconverge} when $\epsilon \rightarrow 0$, we can deduce that \begin{equation}
		    n^\epsilon \rightarrow n^0 := h(p^0) \text{ strongly in } L^2(0,T;L^2(\Omega))
		    \label{eqn:n0}
		\end{equation}
		
		{\bf Step 2:} Now, let us focus on the behavior on the boundary of the domain. Let the linear operator $\gamma$ be the trace operator on the boundary $(0,T) \times \partial \Omega$. For any $\epsilon\in (0,\epsilon_0)$ small enough, we have $\gamma (p^\epsilon) = p^\epsilon \mid_{(0,T) \times \partial \Omega}$, then by the trace theorem, one has 
		\begin{center}
		    $|| \gamma(p^{\epsilon})||_{L^2(0,T;L^2(\partial \Omega))} \leq C || p^{\epsilon}||_{L^2(0,T;H^1(\Omega))} $
		\end{center}
		where the constant $C$ only depends on $\Omega$. Then 
		\begin{center}
		    $|| \gamma(p^{\epsilon})||^2_{L^2(0,T;L^2(\partial \Omega))} \leq C^2 \displaystyle \int_{0}^{T} \int_\Omega| p^\epsilon|^2 dx dt + C^2 \displaystyle \int_{0}^{T} \int_\Omega | \nabla p^\epsilon(t,\cdot) |^2 dx dt < \infty,	$
		\end{center}
		due to \cref{Linfty} and \ref{H1}. Hence, we can deduce that $\gamma(p^\epsilon)$ is weakly convergent in $L^2((0,T) \times \partial\Omega)$. Let $\gamma^0 := \displaystyle \lim_{\epsilon \rightarrow 0}\gamma (p^\epsilon)$. For any function $\psi \in C^1(\overline{\Omega})$, and for $i = 1,\dots, d$, by Green's formula one has 
		\begin{center}
		    $ \displaystyle \int_\Omega \partial_i p^\epsilon \psi dx = -\int_\Omega p^\epsilon \partial_i \psi + \int_{\partial \Omega} \psi \gamma(p^\epsilon) \nu_i dS$. 
		\end{center}
		Since $p^\epsilon$ converges weakly to $p^0$ in $H^1(\Omega)$, when $\epsilon \rightarrow 0$, one has 
		\begin{center}
		    $ \displaystyle \int_\Omega \partial_i p^0 \psi dx = -\int_\Omega p^0 \partial_i \psi + \int_{\partial \Omega} \psi \gamma^0 \nu_i dS$. 
		\end{center}
		We can deduce that $\gamma^0 = \gamma(p^0)$. 
		
		{\bf Step 3:} We pass to the limit in the weak formulation of \cref{eqn:pb4}, for any test function $\psi$ such that $\psi \in C^2([0,T] \times \overline{\Omega}), \psi(T,\cdot) = 0$ in $\Omega ,$	one has 
		
		\noindent 
		\begin{tabular}{r l}
			& $- \displaystyle \underset{\text{strong convergence}}{\underbrace{\int_{0}^{T} \int_\Omega p^\epsilon \partial_t\psi dx dt}} \displaystyle + A \underset{\text{weak convergence}}{\underbrace{\int_{0}^{T} \int_\Omega \nabla p^\epsilon \cdot \nabla \psi dx dt}} =\underset{\text{weak convergence}}{\underbrace{\int_{\Omega} p^{\text{init},\epsilon} \psi(0,\cdot) dx}} $\\
			& $\displaystyle - 2\epsilon A \underset{\text{bounded as } \epsilon \rightarrow 0}{\underbrace{\int_{0}^{T} \int_\Omega \frac{\psi}{1 - \epsilon n^\epsilon} \nabla p^\epsilon \nabla n^\epsilon dx dt}} + \underset{\text{strong convergence}}{\underbrace{\int_{0}^{T} \int_\Omega \psi p^\epsilon(1 - p^\epsilon) (F_1 - F_2)(n^\epsilon,p^\epsilon) dx dt}}  $\\ 
			& $ \displaystyle  -DA \underset{\text{weak convergence}}{\underbrace{\int_{0}^{T} \int_{\partial \Omega} (p^\epsilon - p^{\text{ext},\epsilon})\dfrac{1 - \epsilon n^{\text{ext}, \epsilon}}{1 - \epsilon n^\epsilon}dS}} $.
		\end{tabular}\\
		
		\noindent
		The weak convergence of the last term on the boundary is obtained from \cref{Linfty} and \cref{a5}. When $\epsilon < \epsilon_0$, we have $n^{\text{ext},\epsilon}, n^\epsilon$ are uniformly bounded on $(0,T) \times \Omega$ with respect to $\epsilon$, then $\dfrac{1 - \epsilon n^{\text{ext}, \epsilon}}{1 - \epsilon n^\epsilon} $ converges strongly to $1$ when $\epsilon \rightarrow 0$. From the previous step, one has $p^\epsilon|_{\partial \Omega} = \gamma(p^\epsilon) \rightharpoonup \gamma(p^0)$ weakly in $L^2((0,T) \times \partial\Omega)$. Passing to the limit, we obtain that $p^0 \in L^2(0,T;H^1(\Omega))$ is a weak solution of the following problem
		\begin{center}
		    $\begin{cases}
		        \partial_t p^0 - A\Delta p^0 = p^0(1-p^0)(F_1-F_2)(n^0,p^0) & \text{ in } (0,T) \times \Omega, \\
		        p^0(0,\cdot) = p^\text{init} & \text{ in } \Omega \\
		        \frac{\partial p^0}{\partial \nu} = -D(p^0 - p^\text{ext}) & \text{ on } (0,T)\times \partial  \Omega.
		    \end{cases}$
		\end{center}
	Using \cref{eqn:n0}, we can deduce that this problem is a self-contained initial-boundary-value problem. Moreover, since $0$ and $1$ are respectively sub- and super-solutions of this problem, it admits a unique classical solution with values in $[0,1]$. Hence, all the extracted sub-sequences converge to the same limit $p^0$ and $p^0|_{\partial \Omega} = \gamma(p^0)$.  
	\end{proof}
    \section*{Acknowledgments}
	\begin{minipage}{0.15\textwidth}
            \includegraphics[height = 1.2 cm]{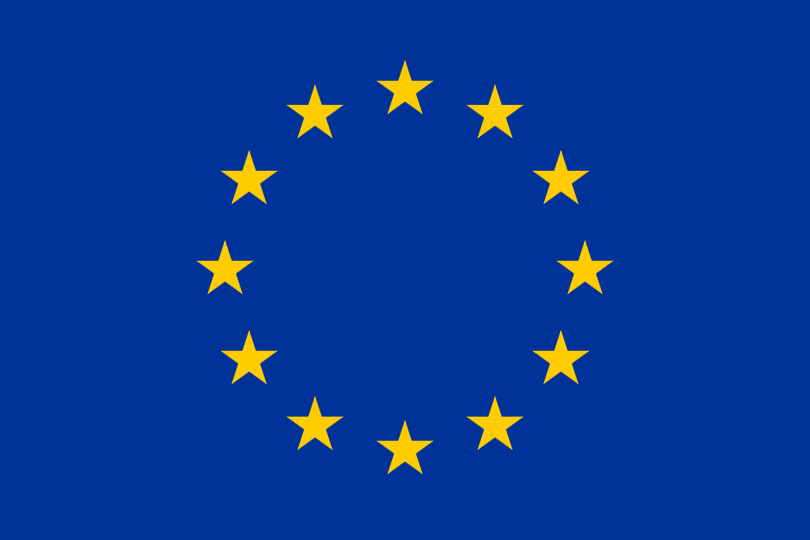} 
        \end{minipage}
        \begin{minipage}{0.8\textwidth}
            {This work has received funding from the European Union's Horizon 2020 research and innovation program under the Marie Sklodowska-Curie grant agreement No 945322.}
        \end{minipage}
\bibliographystyle{acm}
\bibliography{references}
	
\end{document}